\documentclass[letterpaper,11pt]{article}

\newcommand*{\affaddr}[1]{#1} 
\newcommand*{\affmark}[1][*]{\textsuperscript{#1}}
\newcommand*{\email}[1]{\texttt{#1}}

\usepackage[numbers]{natbib}
\usepackage{hyperref} 
\usepackage{amsmath}
\usepackage{amssymb} 
\usepackage{amsthm}
\usepackage{geometry} 
\usepackage{enumitem}
\usepackage{graphicx}
\usepackage{float}  
\usepackage{subcaption} 
\usepackage{color}  
\usepackage{algorithm}
\usepackage{algorithmic}
\usepackage[normalem]{ulem} 
\usepackage[bottom]{footmisc}
\usepackage[export]{adjustbox}

\usepackage{todonotes}

\providecommand{\keywords}[1]
{
  \small	
  \textbf{\textit{Keywords---}} #1
}

\def\grad{\nabla}

\def\ba{\mathbf{a}}
\def\bb{\mathbf{b}}

\def\bd{\mathbf{d}}

\def\bg{\mathbf{g}}

\def\br{\mathbf{r}}
\def\bs{\mathbf{s}}
\def\bt{\mathbf{t}}
\def\bu{\mathbf{u}}

\def\bw{\mathbf{w}}
\def\bx{\mathbf{x}}  
\def\by{\mathbf{y}}
\def\bz{\mathbf{z}}
\def\bA{\mathbf{A}}

\def\bD{\mathbf{D}}

\def\bI{\mathbf{I}}

\def\bY{\mathbf{Y}}
\def\bX{\mathbf{X}}

\def\bo{\mathbf{0}}

\def\m{{\boldsymbol{\mu}}}
\def\n{{\boldsymbol{\nu}}}

\def\th{{\boldsymbol{\theta}}}

\def\b{{\boldsymbol{\beta}}}

\def\cB{\mathcal{B}}

\def\cD{\mathcal{D}}
\def\cE{\mathcal{E}}

\def\cG{\mathcal{G}}

\def\cL{\mathcal{L}}

\def\cN{\mathcal{N}}
\def\cO{\mathcal{O}}
\def\cP{\mathcal{P}}
\def\cQ{\mathcal{Q}}

\def\cS{\mathcal{S}}

\def\mB{\mathbb{B}}

\def\mR{\mathbb{R}}

\def\smskip{\smallskip}

\def\texitem#1{\par\smskip\noindent\hangindent 25pt
               \hbox to 25pt {\hss #1 ~}\ignorespaces}

\def\norm#1{\|#1\|}

\newcommand{\BEAS}{\begin{eqnarray*}}
\newcommand{\EEAS}{\end{eqnarray*}}
\newcommand{\BEA}{\begin{eqnarray}}
\newcommand{\EEA}{\end{eqnarray}}
\newcommand{\BEQ}{\begin{eqnarray}}
\newcommand{\EEQ}{\end{eqnarray}}
\newcommand{\BIT}{\begin{itemize}}
\newcommand{\EIT}{\end{itemize}}
\newcommand{\BNUM}{\begin{enumerate}}
\newcommand{\ENUM}{\end{enumerate}}

\newcommand{\BA}{\begin{array}}
\newcommand{\EA}{\end{array}}

\newif\ifpagenumbering
\pagenumberingtrue

\pagenumberingfalse

\DeclareMathOperator*{\supp}{supp}
\DeclareMathOperator*{\argmin}{argmin}

\def\fprod#1{\left\langle#1\right\rangle}
\def\prox#1{\mathbf{prox}_{#1}}

\newtheorem{theorem}{Theorem}[section]

\newtheorem{lemma}{Lemma}[section]
\newtheorem{remark}{Remark}
\newtheorem{definition}{Definition}[section]

\begin{document}

\date{\today}
\title{\Large\bfseries A first-order optimization algorithm for statistical learning with hierarchical sparsity structure}

%

\author{%
Dewei Zhang\affmark[1], Yin Liu\affmark[1], Sam Davanloo Tajbakhsh\affmark[1]\thanks{Corresponding author}
\\
\email{\{zhang.8705,liu.6630,davanloo.1\}@osu.edu}\\
\\
\affaddr{\affmark[1]Department of Integrated Systems Engineering}\\
\affaddr{The Ohio State University}%
}

\maketitle

\vspace{-0.5cm}
\begin{abstract}
In many statistical learning problems, it is desired that the optimal solution conforms to an a priori known sparsity structure represented by a directed acyclic graph. Inducing such structures by means of convex regularizers requires nonsmooth penalty functions that exploit group overlapping. Our study focuses on evaluating the proximal operator of the Latent Overlapping Group lasso developed by \citet{jacob2009group}. We implemented an Alternating Direction Method of Multiplier with a sharing scheme to solve large-scale instances of the underlying optimization problem efficiently. In the absence of strong convexity, global linear convergence of the algorithm is established using the error bound theory. More specifically, the paper contributes to establishing primal and dual error bounds when the nonsmooth component in the objective function does not have a polyhedral epigraph. We also investigate the effect of the graph structure on the speed of convergence of the algorithm. Detailed numerical simulation studies over different graph structures supporting the proposed algorithm and two applications in learning are provided.
\end{abstract}



\section{Introduction}
Convex sparsity-inducing regularization functions play an important role in different fields including machine learning, statistics, and signal processing~\citep{hastie15statistical}. Some well-known regularizers e.g. \emph{lasso}~\citep{tibshirani1996regression} or \emph{group lasso}~\citep{yuan2006model} are commonly used in different learning frameworks to induce sparsity which allows simultaneous model fitting and feature selection. In contrast to lasso which assumes no a priori knowledge on sparsity pattern, group lasso assumes that variables belong to a priori known groups and the variables within a group tend to affect response similarly, i.e., all are simultaneously zero or nonzero. This introduced more elaborate forms of zero/nonzero patterns, known as \emph{structured sparsity}, to the literature~\citep{bach2012structured}. The focus of this paper is to address structured sparsities represented by a Directed Acyclic Graph (DAG) and the convex Latent Overlapping Group (LOG) lasso regularizer of \cite{jacob2009group} to induce such structure. To be more specific, we develop an optimization framework that allows incorporating this regularizer in large-scale learning problems for huge DAGs. In the remainder of this section, we discuss hierarchical structured sparsity following a DAG, convex regularizers to induce hierarchical sparsity structures, and the proximal mapping of the LOG lasso regularizer.

\subsection{Hierarchical structured sparsity}
Let $\cD=(\cS,\cE)$ be a DAG where $\cS=\{s_1,...,s_N\}$ is the index set of nodes, and $\cE$ is the set of ordered pairs of node indices with an edge from the first to second element, e.g. $(s_i,s_{i'})$ is an edge from $s_i$ to $s_{i'}$. Furthermore, let each node $i$ of the graph contains a set of $d_i$ variables where their indices are contained in $s_i$. We will refer to the variables in node $i$ by $\b_{s_i}$. 

The graph structure contains the sparsity structure between groups of variables in each node. For instance, consider the following two DAGs and assume
\begin{figure}[htbp]
    \centering
    \begin{subfigure}[b]{0.3\textwidth}
        \centering
        \includegraphics[width=0.2\textwidth]{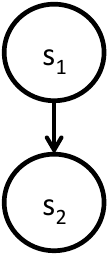}
        \caption{A DAG with two nodes}
        \label{fig:simple0}
    \end{subfigure}
    \qquad
    \begin{subfigure}[b]{0.3\textwidth}
        \centering
        \includegraphics[width=0.5\textwidth]{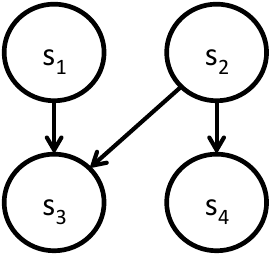}
        \caption{A DAG with four nodes}
        \label{fig:simple1}
    \end{subfigure}
    \caption{Two Directed Acyclic Graphs (DAG)}
    \label{fig:simple}
\end{figure}
the variables in a single node will be all simultaneously non-zero or all zero. The sparsity hierarchy introduced by the graph in Figure~\ref{fig:simple0} is $(s_1=0)\Rightarrow(s_2=0)$ and $(s_2\neq 0)\Rightarrow(s_1\neq 0)$. Note that $s_j=0$ means that all of the variables in $s_j$ are equal to zero; similarly, $(s_j\neq 0)$ means all of the variables in $s_j$ are nonzero. However, there are scenarios where a node has more than one ancestor, e.g. node $s_3$ in Figure~\ref{fig:simple1}. Such scenarios can potentially be interpreted in two different ways. Under \emph{strong hierarchy} assumption, all of the immediate ancestor nodes need to be nonzero for their descendent node to be nonzero, e.g., in Figure~\ref{fig:simple1}, $(s_3\neq 0)\Rightarrow (s_1\neq 0, s_2\neq 0)$, and $(s_1=0 \text{ or } s_2=0)\Rightarrow (s_3=0)$. Under \emph{weak hierarchy} assumption, for a descendent node to be nonzero it suffices that any of its immediate ancestors be nonzero, e.g., in Figure~\ref{fig:simple1}, $(s_3\neq 0)\Rightarrow (s_1\neq 0 \text{ or } s_2\neq 0)$, and $(s_1=0, s_2=0)\Rightarrow (s_3=0)$ \citep{bien2013lasso}.

We are interested in statistical learning problems that require their solutions to follow given sparsity structures in form of DAGs. To be more specific, given a DAG $\cD$, there is a learning problem of the form
\begin{equation}\label{eq:main1}
\min_{\b} \quad \big\{ \cL(\b) \quad \text{s.t.} \quad \b\in\cB, \ \supp(\b)\in\cD \big\}
\end{equation}
where $\cL:\mR^d\rightarrow\mR$ is a smooth, convex or nonconvex loss function with $d=\sum_{i=1}^N d_i$ being the problem dimension, $\cB\subseteq\mR^d$ is a closed set, and with a slight abuse of notation, $\supp(\b)\in\cD$ denotes that the support of $\b$ (index set of its nonzero elements) follows $\cD$ in the \emph{strong} sense.  
One way to formulate $\supp(\b)\in\cD$, explicitly, is by introducing binary variables. For instance, assuming one variable per node, formulating the hierarchy in Figure~\ref{fig:simple0} shall be performed as
\begin{align*}
\begin{split}
z\epsilon\leq|\beta_1|,\quad |\beta_2|\leq z\mu, \quad z\in\{0,1\},
\end{split}
\end{align*}
where $\epsilon$ and $\mu$ are reasonably small and large numbers, respectively. Introduction of binary variables makes the optimization problem a Mixed Integer Program (MIP)~\citep{wolsey2014integer} - see also~\cite{bach2013learning,bach2010structured}. Finding the global optimal solution of large-scale MIPs for large DAGs is generally computationally challenging. We, however, would like to note some significant advances in MIP algorithms for statistical learning, specifically for feature selection -- see e.g. \cite{bertsimas2016best,manzour2019integer,atamturk2018strong,mazumder2017thediscrete,bertsimas2017sparse,bertsimas2019unified}. Similar to using $\ell_1$ norm as a convex approximation to $\ell_0$ (pseudo) norm to induce sparsity, there are convex regularizers that promote hierarchical sparsity structures. Needless to mention, these approximation methods do not guarantee exact conformance of their solutions to given hierarchies, but, they allow solving high-dimensional problems. 

\subsection{Group Lasso with overlaps vs. Latent Overlapping Group lasso}
There are mainly two convex regularizers to introduce hierarchical structured sparsity: 1. Group Lasso (GL) 2. Latent Overlapping Group Lasso (LOG)~\citep{yan2017hierarchical}. Given a set of groups of variables $\cG$, the GL regularizer is defined as
\begin{equation}\label{eq:GL}
\Omega_{\text{GL}}(\b)=\sum_{g\in\cG} w_g\norm{\b_g}
\end{equation}
where $w_g$ is a positive weight corresponding to group $g$ and $\b_g\in\mR^{|g|}$ is equal to $\b$  for elements whose indices belongs to $g$ and zero for other elements, and the $\norm{\cdot}$ is either an $\ell_2$ or $\ell_\infty$ norm. To induce hierarchical sparsity structure using the GL penalty, the groups should be defined in a descendants form, for instance, for the graph in Figure~\ref{fig:simple1} the groups should be $\cG=\{s_3,s_4,\{s_1,s_3\},\{s_2,s_3,s_4\}\}$ where $s_3=\text{descendants}(\cD;s_3)$, $\{s_1,s_3\}=\text{descendants}(\cD;s_1)$, $s_4=\text{descendants}(\cD;s_4)$, and $\{s_2,s_3,s_4\}=\text{descendants}(\cD;s_2)$. The group lasso sets to zero a union of a subset of groups introduced in $\cG$. However, since there are overlaps between the groups defined in $\cG$, the support of the solution induced by GL is not necessarily a union of the groups. This is because of the fact that the complement of a union of a subset of groups is not necessarily a union of groups.

As an alternative to GL, \citet{jacob2009group} introduced LOG regularizer which is defined as
\begin{equation}\label{eq:LOG}
\Omega_{\text{LOG}}(\b)=\inf_{\n^{(g)},\ g\in\cG} \left\{\sum_{g\in\cG}w_g\norm{\n^{(g)}}_2 \ \ \text{s.t.} \ \sum_{g\in\cG}\n^{(g)}=\b,\ \n^{(g)}_{g^c}=0 \right\}
\end{equation} 
which sets to zero a subset of groups. Since $\b$ is the sum of latent variables $\n^{(g)}\in\mR^d$, its support is the union of the groups of nonzero latent variables. Given a DAG $\cD$ with $N$ nodes, there exist $N$ groups in $\cG$ (i.e., $N=|\cG|$). To induce a hierarchical sparsity using the LOG penalty the group corresponding to each node contains the node indices of all its ancestors, i.e., $\cG=\text{ancestors}(\cD)$. For instance, the group set for the graph in Figure~\ref{fig:simple1} is $\cG=\{s_1,s_2,\{s_1,s_2,s_3\},\{s_2,s_4\}\}$ where $s_1=\text{ancestors}(\cD;s_1)$, $s_2=\text{ancestors}(\cD;s_2)$, $\{s_1,s_2,s_3\}=\text{ancestors}(\cD;s_3)$, and $\{s_2,s_4\}=\text{ancestors}(\cD;s_4)$. Figure~\ref{fig:tree_3} shows a simple tree with three nodes, and the ancestor grouping scheme; Figure~\ref{fig:log_tree_3} shows the latent variables within the constraint in the LOG penalty.
\begin{figure}[htbp]
    \centering
    \begin{subfigure}[b]{0.4\textwidth}
        \centering
        \includegraphics[width=0.45\textwidth]{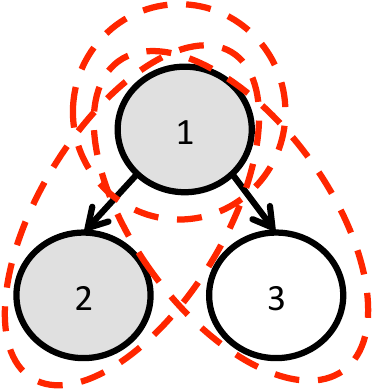}
        \caption{A tree with three nodes. Red dashed lines show the groups.}
        \label{fig:tree_3}
    \end{subfigure}
    \qquad
    \begin{subfigure}[b]{0.5\textwidth}
        \centering
        \includegraphics[width=0.8\textwidth]{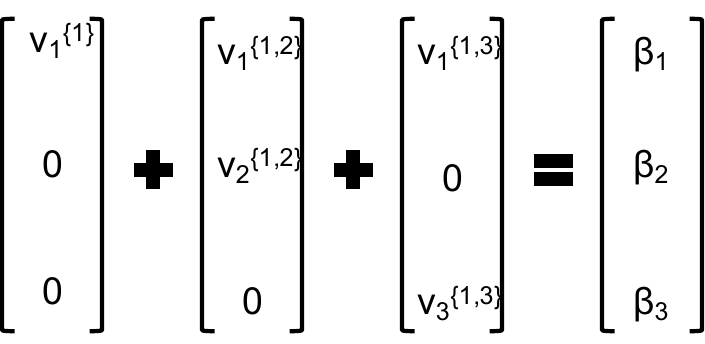}
        \caption{The constraint within the LOG penalty corresponding to the tree in \ref{fig:tree_3}}
        \label{fig:log_tree_3}
    \end{subfigure}
    \caption{LOG penalty and the required groups to induce a tree structure}\label{fig:tree_log}
\end{figure}
Recently, \citet{yan2017hierarchical} performed a detailed comparison of GL vs. LOG regularizers. They showed that compared to LOG, GL sets to zero parameters which are deeper in the hierarchy.  Hence, for DAGs with deep hierarchies it is very probable that GL sets to zero deeper variables which is undesirable. Furthermore, with LOG penalty, one has control over the solution support as it is a subset of columns of latent variables. 

In the next section, we discuss solving statistical learning problems in the regularized form using the nonsmooth LOG penalty and propose solving them using proximal methods.

\subsection{Proximal operator of the LOG penalty}
Given a hierarchical sparsity structure represented by a graph $\cD$, an approximate convex optimization  problem to \eqref{eq:main1} is 
\begin{equation}\label{eq:main_convex}
\min_{\b} \quad \big\{ \cL(\b) +\lambda\Omega_{\text{LOG}}(\b) \quad \text{s.t.} \quad \b\in\cB \big\}
\end{equation}
where $\Omega_{\text{LOG}}(.)$ is the LOG penalty introduced in \eqref{eq:LOG} with appropriately chosen groups, and $\lambda> 0$ is a parameter that controls the tradeoff between the loss function and the penalty. Indeed, problem \eqref{eq:main_convex} is a convex \emph{nonsmooth} program; hence,  proximal methods are suitable to solve large instances of this problem~\citep{nesterov2013gradient,beck2009fast,parikh2014proximal}.

Similar to gradient methods that require iterative evaluation of the gradient, proximal methods require iterative evaluation of the proximal operator~\citep{parikh2014proximal}. The proximal operator of a function $\lambda\Omega(\b)$ in general ($\lambda\Omega_{\text{LOG}}(\b)$ in this case) evaluated at $\bb\in\mR^d$ is defined as 
\begin{equation}\label{eq:prox_def}
\prox{\lambda\Omega}(\bb)\triangleq \argmin_{\b\in\mR^d}\left\{\lambda\Omega_{\text{LOG}}(\b)+\frac{1}{2}\norm{\b-\bb}_2^2\right\}.
\end{equation}
Using the definition of $\Omega_{\text{LOG}}$, evaluating the proximal operator of the LOG penalty requires solving 
\begin{equation}\label{eq:prox_LOG}
\min_{\n^{(g)}\in\mR^d} \left\{\lambda\sum_{g\in\cG} w_g\norm{\n^{(g)}}_2+\frac{1}{2}\norm{\sum_{g\in\cG}\n^{(g)}-\bb}_2^2 \ \ \text{s.t.} \ \ \n^{(g)}_{g^c}=0 ,\ \forall g\in\cG \right\}.
\end{equation}
over the latent variables $\n^{(g)},\ g\in\cG$. A classical method to solve \eqref{eq:prox_LOG} is some implementation of the proximal Block Coordinate Descent (BCD) algorithm~\citep{villa2014proximal} -- see e.g. Algorithm~\ref{alg:BCD_LOG}.
\begin{algorithm}[htbp]
\caption{Block Coordinate Descent (BCD) to solve $\prox{\lambda\Omega_{\text{LOG}}}(\bb)$}
\label{alg:BCD_LOG}
\begin{algorithmic}[1]
\REQUIRE $\bb, \lambda, \bw, \cG$
\STATE $\b=\bo$
\STATE $\n^{(g)}=0, \ \forall g\in\cG$
\WHILE{\text{stopping criterion not met}}
\FOR{$g\in\cG$}
\STATE $\b\gets\b-\n^{(g)}$
\STATE $\n^{(g)}\gets\cS_G(\bb_g-\b_g,\lambda w_g)\triangleq (\bb_g-\b_g)\max\{1-\frac{\lambda w_g} {\norm{\bb_g-\b_g}},0\}$,
\STATE $\b\gets\b+\n^{(g)}$
\ENDFOR
\ENDWHILE\\
\hspace{-0.8cm }\textbf{Output:} $\b$
\end{algorithmic}
\end{algorithm}
Provided an algorithm to evaluate $\prox{\lambda\Omega_{\text{LOG}}}$ \emph{efficiently}, one may use a proximal optimization method e.g. proximal gradient method~\cite{beck2009fast,nesterov2013gradient,parikh2014proximal} to solve \eqref{eq:main_convex}. The main challenge is to evaluate the proximal operator of the LOG penalty \eqref{eq:prox_def} for large DAGs with large $|\cG|$. The main drawbacks of the BCD algorithm to solve \eqref{eq:prox_def} are as follows. First, even though convergence of the BCD algorithm for nondiffrentiable but separable functions has been established~\cite{tseng2001convergence}, to the best of our knowledge, the convergence rate of the algorithm for nonsmooth optimization is \emph{sublinear}~\cite{tseng2009coordinate,razaviyayn2013unified}. Second, the BCD algorithm follows a Gauss-Siedel update rule and hence it cannot be parallelized. In this work, we introduce an efficient first-order method, based on Douglas-Rachford operator splitting, that can solve \eqref{eq:prox_def} over large graphs with fast, i.e., \emph{linear}, rate of convergence.

\begin{remark}
The intent of this work is to propose an efficient optimization algorithm to solve the proximal mapping of the LOG penalty \eqref{eq:prox_def}. The resulting standalone algorithm can then be embedded in any prox-based optimization algorithm to solve \eqref{eq:main1} for convex or nonconvex settings. Hence, solving \eqref{eq:main1} is not the purpose of our paper which is the reason why we do not specify any structures on $\cL(\cdot)$ or $\cB$, e.g., convexity or etc.
\end{remark}

\subsection{Contributions}
We propose an ADMM algorithm with a sharing scheme to solve $\prox{\Omega_{\text{LOG}}}$ defined over large DAGs. The underlying DAG may have any general structure (e.g. not necessary to be a path graph) and the algorithm is guaranteed to converge to its optimal solution. Furthermore, the computationally challenging subproblem of the algorithm (step 4 in Algorithm~\ref{alg:sharing_Prox_LOG}) can be run fully in parallel.

We proved linear convergence of the algorithm given a sufficiently small stepsize \emph{in the absence of strong convexity}. Establishing the linear convergence rate is based on the error bound theory (see e.g.\cite{tseng2010approximation,luo1993convergence,zhang2013linear,hong2017linear}), and our contributions are as follows:

1. The dual error bound is established in the presence of $\ell_2$-norm in the nonsmooth component of the objective function. A common key assumption in previous works \emph{requires the nonsmooth component to have a polyhedral epigraph}. Our proof shows an approach to escape such assumption and enables further extension. 

2. On the primal side, we rigorously prove the error bound for the augmented Lagrangian function. The main challenges are the presence of the dual variable $\by$ and the splitting of $\bx$ into two blocks , i.e., $\bx^1$ and $\bx^2$, which requires a number of technical details, e.g., comparing the limiting behavior of $\bx^{1,k}$ and $\bx^{2,k}$. To the best of our knowledge, this is the first work showing error bound for augmented Lagrangian function with block structured variables. 

3. As an auxiliary result needed to show the linear rate of convergence of the algorithm, and by using the proof in \cite{glowinski}, we formally show the uniform boundedness (with respect to the stepsize) of both primal and dual updates generated by the algorithm.

4. We looked into the effect of the graph structure on the $O(1)$ constant of the convergence rate; furthermore, we performed detailed numerical experiments on six different graphs comparing the proposed method against five other state-of-the-art optimization techniques for this problem.

\subsection{Related work}
As discussed above, hierarchical sparsity structures are generally enforced by either introducing constraints to the underlying optimization problem or adding regularizers to its objective function. We discuss the relevant work in the later category with convex regularizers which are directly related to our work. 

\citet{zhao2006model} proposed composite absolute penalty to express grouping and hierarchical structures and used a stage-wise lasso algorithm to approximate the regularization path for convex problems. \citet{radchenko2010variable} proposed an adaptive nonlinear interaction structure for least-square regression with two-way interaction for high-dimensional problems. \citet{schmidt2010convex} enforced the hierarchical constraints using grouped $\ell_1$ regularization with overlapping groups for log-linear models of any order and proposed an active set method to solve the underlying problem. \citet{haris2016convex} proposed a general framework for variable selection problems regularized with group lasso with overlaps and proposed to solve the underlying problem by the ADMM algorithm. \citet{jenatton2011structured} explored the relationship between the groups defining the norm and the resulting nonzero patterns and provided forward and backward algorithms to go back and forth between groups and patterns. Furthermore, \citet{jenatton2011proximal} considered finding the proximal mapping of the group lasso penalty through its dual problem and showed that a BCD algorithm solves the dual in one pass for tree graphs. \citet{mairal2011convex} showed that the proximal operator of the overlapping group lasso under $\ell_\infty$ norm can be computed in polynomial time by solving a quadratic min cost-flow problem and used proximal splitting to solve the problem in higher dimensions. \citet{she2018group} considered group lasso with overlap penalty and provide the minimax lower bounds for strong and weak hierarchical models and showed their proposed estimators enjoy sharp rate oracle inequalities. \citet{hazimeh2019learning} proposed a scalable algorithm based on proximal gradient descent to solve the underlying problem. Their method enjoys a proximal screening step that identifies many of the zero variables and groups and also allows to solve the problem in parallel; furthermore, they proposed an efficient active set method based on gradient screening. 

\citet{jacob2009group}, first, proposed the LOG penalty and studied its theoretical properties - see also \citet{obozinski2011group} for theoretical discussions on choice of the weights. \citet{villa2014proximal} proposed accelerated proximal method and proved the convergence of the overall learning problem for the least-square loss function. Furthermore, they developed an active set method to compute the inner proximal mapping relatively fast. \citet{lim2015learning} imposed strong hierarchy through a constrained optimization problem for which they found an equivalent unconstrained problem. \citet{chouldechova2015generalized} incorporated a LOG-like penalty to fit generalized additive models. Finally, \citet{yan2017hierarchical}  compared statistical properties of GL and LOG lasso penalties. They also proposed a finite-step  algorithm to compute the proximal operator of the LOG penalty for path graphs and extended it to general DAGs with an ADMM framework. However, their extension for general DAGs (1) highly depends on how DAG is decomposed into different path graphs, and (2) it lacks theoretical convergence and convergence rate analysis. Compared to their method, we formulated the proximal map for the original DAG directly into an ADMM framework with a sharing scheme and established global linear convergence of the algorithm. We compare the convergence time of our algorithm with their path-based ADMM in Section~\ref{sec:simulations}. 

\vspace{0.2cm}
\noindent\textbf{Notation.} Vectors are denoted by lowercase bold letter while matrices are denoted by uppercase letters. The identity matrix is denoted by $\bI$. Let $\cG$ be a set, then its cardinality is denoted by $|\cG|$. Given a vector $\b\in\mR^d$ and $\bg\subseteq\{1,...,d\}$, $\b_{\bg}\in\mR^{|\bg|}$ subsets $\b$ over the set $\bg$. We denote the $j$-th column of the matrix $A$ by $A_{.j}$, similarly we denote its $i$-th row by $A_{i.}$. Furthermore, similar to the vector case, if $\bg\subseteq\{1,...,n\}$, then $A_{.g}\in\mR^{m\times|g|}$ subsets $A$ over the columns indexed by $\bg$. Inner product of two vectors is defined as $\fprod{\ba,\bb}=\ba^\top\bb$.

\vspace{0.2cm}
The rest of the paper is organized as follows. In Section~\ref{sec:prox_LOG}, we propose an ADMM algorithm with the sharing scheme to evaluate the proximal operator of the LOG penalty. Section~\ref{sec:convergence} provides detailed convergence analysis of the proposed algorithm while proofs are relegated to the Appendix. Section~\ref{sec:numerical} provides some numerical simulation studies confirming our theoretical complexity bound and compares the ADMM algorithm with other methods. Furthermore, Section~\ref{sec:numerical} contains two applications that use LOG penalty to induce sparsity structure related to topic modeling and breast cancer classification. Finally, Section~\ref{sec:conclusion} provides some concluding remarks.

\vspace{0.2cm}
\section{Evaluating the proximal operator of the LOG}\label{sec:prox_LOG}
Consider the optimization problem \eqref{eq:prox_LOG} to find the proximal operator of the LOG penalty. Given a general DAG $\cD$, each iteration of the BCD algorithm requires updating $\n^{(g)},\ \forall g\in\cG$. If $|\cG|$ is a large number, then per iteration complexity of the algorithm $\cO(\sum_{g\in\cG}|g|)$ is costly. Furthermore, to the best of authors knowledge, convergence rates of the BCD algorithm for general nonsmooth optimization problems have not been well studied. However, convergence of the algorithm for problems where the nondiffrentiable part is separable is established, cf. \cite{tseng2001convergence}.  

In this paper, we develop an Alternating Direction Method of Multiplier (ADMM) to solve the proximal operator of the LOG penalty. The algorithm is parallelizable which makes it suitable when the number of groups is very large. Define $\bx=[\n^{(g)}_g]_{g\in\cG}\in\mR^n$, $n=\sum_{g\in\cG}|g|$, be a long vector that contains the nonzero elements of $\n^{(g)},\ g\in\cG$, for some random order $\cP$ of the groups. Furthermore, let $j(\cdot):\ \cG\rightarrow\{1,...,n\}$ be the set map that associates a group $g\in\cG$ to its indices in vector $\bx$ given an order of the groups. For instance, for $\cG=\{\{1\},\{1,2\},\{1,3\}\}$ ordered as $\cP$ from left-to-right: 1. $n=5$, 2. $j_{\cP}(\{1\})=\{1\}$, $j_{\cP}(\{1,2\})=\{2,3\}$, and $j_{\cP}(\{1,3\})=\{4,5\}$. To simplify the notation, the group ordering $\cP$ is omitted. Finally, \eqref{eq:prox_LOG} can equivalently be written as
\begin{equation}
\label{eq:main}
\min_{\bx\in\mR^n} \lambda\sum_{g\in\cG}\norm{W^{g}\bx}_2+\frac{1}{2}\norm{M\bx-\bb}_2^2,
\end{equation}
where $W^{g}=w_gU^{j(g)}$, $U^{j(g)}\in\mR^{|g|\times n}$ such that $[{U^{j(g)}}^\top]_{g\in\cG}=\bI\in\mR^{n\times n}$, and $M\in\mB^{d\times n}$ sums elements of $\bx$ along each coordinate. Or, equivalently, \eqref{eq:prox_LOG} can be written as
\begin{equation}
\label{eq:main_2}
\min_{\bx\in\mR^n} f(\bx)\triangleq\lambda\sum_{g\in\cG}w_g\norm{\bx_{j(g)}}_2+\frac{1}{2}\norm{M\bx-\bb}_2^2.
\end{equation}
Problem \eqref{eq:main_2} is a convex (\emph{not} strongly convex since $M$ is not a full-column rank matrix), nonsmooth optimization program. This problem can also be solved using the proximal gradient method~\cite{zhang2013linear,tseng2010approximation}. We propose to solve this problem using the Alternating Direction Method of Multipliers (ADMM). First, splitting the problem into two blocks~\citep{boyd2011distributed}, we have
\begin{equation}
\label{eq:main_split}
\min_{\bx^1,\bx^2\in\mR^n} \left\{ F(\bx^1,\bx^2)\triangleq\lambda\sum_{g\in\cG}w_g\norm{\bx^1_{j(g)}}_2+\frac{1}{2}\norm{M\bx^2-\bb}_2^2, \ \ \text{s.t.} \ \ \bx^1=\bx^2 \right\}.
\end{equation}
The augmented Lagrangian function for \eqref{eq:main_split} is 
\begin{equation}
\label{eq:augLagrange}
L_{\rho}(\bx^1,\bx^2;\by)=\lambda\sum_{g\in\cG}w_g\norm{\bx^1_{j(g)}}_2+\frac{1}{2}\norm{M\bx^2-\bb}_2^2+\fprod{\by,\bx^1-\bx^2}+\frac{\rho}{2}\norm{\bx^1-\bx^2}_2^2,
\end{equation}
where $\by\in\mR^n$ is the Lagrange multiplier for the linear constraint $\bx^1=\bx^2$, and $\rho\geq 0$ is a constant. Furthermore, the augmented dual function is given by
\begin{equation}
\label{eq:augDual}
g_{\rho}(\by)=\min_{\bx^1,\bx^2\in\mR^n} \lambda\sum_{g\in\cG}w_g\norm{\bx^1_{j(g)}}_2+\frac{1}{2}\norm{M\bx^2-\bb}_2^2+\fprod{\by,\bx^1-\bx^2}+\frac{\rho}{2}\norm{\bx^1-\bx^2}_2^2,
\end{equation}
which results into the the dual problem
\begin{equation}
\label{eq:dualProblem}
\max_{\by\in\mR^n} \ \ g_{\rho}(\by).
\end{equation}
The ADMM iterates in the unscaled form~\cite{boyd2011distributed} are 
\begin{align}
\bx_{j(g)}^{1,k+1} &\gets \argmin_{\bx^1_{j(g)}\in\mR^{|g|}} \lambda w_g\norm{\bx_{j(g)}^1}_2+\frac{\rho}{2}\norm{\bx^1_{j(g)}-\bx^{2,k}_{j(g)}+\frac{1}{\rho}\by^k_{j(g)}}_2^2,\ \ \forall g\in\cG,  \label{eq:admm_0_1} \\
\bx^{2,k+1} &\gets \argmin_{\bx^2\in\mR^n} \frac{1}{2}\norm{M\bx^2-\bb}_2^2+\frac{\rho}{2}\norm{\bx^2-\bx^{1,k+1}-\frac{1}{\rho}\by^k}_2^2, \label{eq:admm_0_2} \\
\by^{k+1}_{j(g)} & \gets \by^k_{j(g)}+\alpha\large(\bx^{1,k+1}_{j(g)}-\bx^{2,k+1}_{j(g)}\large),\ \ \forall g\in\cG,  \label{eq:admm_0_3} 
\end{align}
where $\alpha$ is the dual stepsize. Algorithm~\ref{alg:ADMM_Prox_LOG_unscaled} illustrates the resulting (unscaled) ADMM algorithm to evaluate the proximal map of the LOG penalty. Note that the subproblem \eqref{eq:admm_0_1} is parallelizable across groups, and the solution to each subproblem is available in the closed from. However, even though the update \eqref{eq:admm_0_2} has a closed-form solution, it involves inverting an $n\times n$ matrix which generally requires $\cO(n^3)$ operations, per iteration. Hence, for large DAGs where $|\cG|$ and hence $n$ is large, the second update is very slow. 
\begin{algorithm}[htbp]
\caption{ADMM to solve $\prox{\lambda\Omega_{\text{LOG}}}(\bb)$ in the unscaled form}
\label{alg:ADMM_Prox_LOG_unscaled}
\begin{algorithmic}[1]
\REQUIRE $\bb, \lambda, \alpha, w_g \ \forall g\in\cG, j(.):\cG\rightarrow[n]$
\STATE $k=0,\ \by^0=\bo,\ \bx^{2,0}=\bo$
\WHILE{\text{stopping criterion not met}}
\STATE $k\gets k+1$ \\
\STATE $\bx_{j(g)}^{1,k+1} \gets \prox{\lambda w_g\norm{\cdot}_2}(\bx^{2,k}_{j(g)}-\frac{1}{\rho}\by^k_{j(g)}),\ \ \forall g\in\cG$ \\
\STATE $\bx^{2,k+1} \gets (M^\top M+\rho\bI)^{-1}(M^\top\bb+\rho\bx^{1,k+1}+\by^k)$ \\
\STATE $\by^{k+1}_{j(g)} \gets \by^k_{j(g)}+\alpha\large(\bx^{1,k+1}_{j(g)}-\bx^{2,k+1}_{j(g)}\large),\ \ \forall g\in\cG$ \\
\ENDWHILE\\
\STATE $\b=M\bx^{1,k}$ \\
\hspace{-0.7cm }\textbf{Output:} $\b$
\end{algorithmic}
\end{algorithm}

To deal with this issue, below we propose a sharing scheme that helps solving the second subproblem efficiently. First, we put the variables in the matrix form. Define $X\in\mR^{d\times|\cG|}$ be a matrix that stacks $\n^{(g)}, \ g\in\cG$ where its columns are indexed by $g\in\cG$. Problem \eqref{eq:prox_LOG} can be written in the matrix form as
\begin{equation}
\label{eq:main_3}
\min_{X\in\mR^{d\times|\cG|}} \left\{\lambda\sum_{g\in\cG}w_g\norm{X_{.g}}_2+\frac{1}{2}\norm{\sum_{g\in\cG} X_{.g}-\bb}_2^2, \ \text{s.t.} \ (X_{.g})_{g^c}=\bo~\forall g\in\cG \right\}.
\end{equation}
Splitting the problem into two blocks, the problem is equivalent to
\begin{equation}
\label{eq:admm_2}
\min_{X^1,X^2\in\mR^{d\times|\cG|}} \left\{\lambda\sum_{g\in\cG}w_g\norm{X^1_{.g}}_2+\frac{1}{2}\norm{\sum_{g\in\cG}X^2_{.g}-\bb}_2^2, \ \text{s.t.} \ X^1=X^2, \ (X^1_{.g})_{g^c}=\bo~\forall g\in\cG \right\}.
\end{equation}
The ADMM iterates in the \emph{scaled form} (through defining $U:=(1/\rho)Y$ where $Y$ is the dual variable in the matrix form - see \cite{boyd2011distributed}  for details) to solve \eqref{eq:admm_2} are
\begin{align}
X_{.g}^{1,k+1} &\gets \argmin_{X^1_{.g}\in\mR^d} \left\{\lambda w_g\norm{X_{.g}^1}_2+\frac{\rho}{2}\norm{X^1_{.g}-X^{2,k}_{.g}+U^k_{.g}}_2^2\ \ \text{s.t.}\ (X^1_{.g})_{g^c}=\bo\right\},\ \forall g\in\cG, \label{eq:admm2_1} \\
X^{2,k+1} &\gets \argmin_{X^2\in\mR^{d\times|\cG|}} \frac{1}{2}\norm{\sum_{g\in\cG}X_{.g}^2-\bb}_2^2+\frac{\rho}{2}\sum_{g\in\cG}\norm{X^2_{.g}-X^{1,k+1}_{.g}-U^k_{.g}}_2^2,\label{eq:admm2_2}  \\ 
U^{k+1}_{.g} & \gets U^k_{.g}+(\alpha/\rho)\large(X^{1,k+1}_{.g}-X^{2,k+1}_{.g}\large),\ \forall g\in\cG, \label{eq:admm2_3}
\end{align}
Similar to the vector form, the solution to subproblem \eqref{eq:admm2_1} is provided by the proximal map of the $\ell_2$-norm and can be parallelized across groups, see step 4 in Algorithm~\ref{alg:sharing_Prox_LOG}. Subproblem \eqref{eq:admm2_2} is potentially a large problem in $d|\cG|$ variables for a large DAG (equivalent to the second update in the above algorithm in the vector form); however, it is possible to decrease its size to only $d$ variables. Subproblem \eqref{eq:admm2_2} is equivalent to
\begin{equation}
\min_{X^2\in\mR^{d\times|\cG|},~\bar{\bx}^2\in\mR^d} \left\{\frac{1}{2}\norm{~|\cG|\bar{\bx}^2-\bb}_2^2+\frac{\rho}{2}\sum_{g\in\cG}\norm{X^2_{.g}-X^{1,k+1}_{.g}-U^k_{.g}}_2^2\ \text{s.t.} \ \bar{\bx}^2=(1/|\cG|)\sum_{g\in\cG}X^2_{.g} \right\}.
\end{equation}
Minimizing over $X^2_{.g}$ with $\bar{\bx}^2$ fixed and using optimality conditions, we get 
\begin{equation}
\label{eq:xg-fixed-xbar}
X^2_{.g}=\bar{\bx}^2+X^{1,k+1}_{.g}+U^k_{.g}-(1/|\cG|)\sum_{g\in\cG}(X^{1,k+1}_{.g}+U^k_{.g}),\ \forall g\in\cG.
\end{equation}
Using \eqref{eq:xg-fixed-xbar} to solve \eqref{eq:admm2_2} we get
\begin{equation}
\label{eq:x2bar}
\bar{\bx}^2=\frac{1}{|\cG|+\rho}\Big(\bb+\frac{\rho}{|\cG|}\sum_{g\in\cG}(X^{1,k+1}_{.g}+U^k_{.g})\Big).
\end{equation}
Furthermore, using \eqref{eq:xg-fixed-xbar} in \eqref{eq:admm2_3}, we get
\begin{equation}
U_{.g}^{k+1}=U_{.g}^{k}+(\alpha/\rho)\big(\frac{1}{|\cG|}\sum_{g\in\cG}(X^{1,k+1}_{.g}+U^k_{.g})-\bar{\bx}^2-U^k_{.g}\big), \ \forall g\in\cG.
\end{equation}
The sharing implementation of the proposed ADMM algorithm is illustrated in Algorithm~\ref{alg:sharing_Prox_LOG}.
\begin{algorithm}[htbp]
\caption{ADMM to solve $\prox{\lambda\Omega_{\text{LOG}}}(\bb)$ in the scaled form with the sharing scheme}
\label{alg:sharing_Prox_LOG}
\begin{algorithmic}[1]
\REQUIRE $\bb, \lambda, \alpha, w_g \ \forall g\in\cG$
\STATE $k=0,\ \bar{\bu}^0=\bo,\ \bar{\bx}^{2,0}=\bo$
\WHILE{\text{stopping criterion not met}}
\STATE $k\gets k+1$ \\
\STATE $X_{gg}^{1,k+1} \gets \prox{\lambda w_g\norm{\cdot}_2}(X^{1,k}_{gg}+\bar{\bx}^{2,k}_{g}-\bar{\bu}^k_{g}-\bar{\bx}^{1,k}_g), \ \ \forall g\in\cG$  \\
\STATE $X_{g^cg}^{1,k+1} \gets \bo, \ \ \forall g\in\cG$
\STATE $\bar{\bx}^{1,k+1} \gets \frac{1}{|\cG|}\sum_{g\in\cG}X^{1,k+1}_{.g}$
\STATE $\bar{\bx}^{2,k+1}\gets\frac{1}{|\cG|+\rho}\big(\bb+\rho(\bar{\bx}^{1,k+1}+\bar{\bu}^k)\big)$ \\
\STATE $\bar{\bu}^{k+1}=\bar{\bu}^{k}+(\alpha/\rho)\big(\bar{\bx}^{1,k+1}-\bar{\bx}^{2,k+1}\big).$ \\
\ENDWHILE\\
\STATE $\b=\sum_{g\in\cG} X^{1,k+1}_{.g}$ \\
\hspace{-0.7cm } \textbf{Output:} $\b$
\end{algorithmic}
\end{algorithm}

\section{Convergence analysis}\label{sec:convergence}
\citet{Eckstein2012ADMM} showed the iterates generated by two-block ADMM converges to some limiting points under certain conditions. Our setting follows their proposition. It is also straightforward to show such limiting points are optimal solutions. In this section, we establish \emph{linear} convergence rate of the ADMM algorithm to solve the proximal operator of the LOG penalty \eqref{eq:main_2} given a sufficiently small stepsize using the error bound theory.

\subsection{Rate of convergence}\label{sec:rate}
Note that the objective function of \eqref{eq:main_2} is \emph{not} strongly convex since $M$ is not full column rank. To establish the linear convergence rate, we will use the error bound theory which is well-established for primal methods -- see \cite{tseng2010approximation} and references therein. For a dual method, one needs to show that both primal and dual error bounds hold for the problem under investigation. As mentioned in the contributions, showing the dual error bound in the presence of $\ell_2$-norm in the objective function (which results in a second-order cone epigraph) is not trivial. Furthermore, in the absence of a bounded feasible region, boundedness of the iterates needs to be established. Finally, establishing primal error bound in the presence of the dual variable and under 2-block splitting is elaborate.  All of these challenges are addressed in this section.

Let $\bX^*\subseteq\mR^{2n}$ and $\bY^*\subseteq\mR^n$ denote the primal and dual optimal solution sets to \eqref{eq:main_split} and \eqref{eq:dualProblem}, respectively. Let $\bX(\by)\subseteq\mR^{2n}$ denote the optimal solution set to the problem of minimizing the augmented Lagrangian function \eqref{eq:augDual} given $\by\in\mR^n$. Note the augmented Lagrangian \eqref{eq:augLagrange} is strongly convex in $\bx^1$ or $\bx^2$, but not jointly strongly convex. We use $\bx(\by)=({\bx^1}^\top(\by),{\bx^2}^\top(\by))^\top\in\bX(\by)$ to represent a minimizer of \eqref{eq:augLagrange} given $\by$. Let $E\triangleq[\bI,-\bI]\in\mR^{n\times 2n}$ and $M$ as in  \eqref{eq:main_2}, and define
\begin{equation}
\label{eq:ell_func}
\ell(\bx^1,\bx^2)\triangleq \phi_{\bb}(M\bx^2)+\psi_{\rho}(E\bx)
\end{equation}
where the functions $\phi_{\bb}:\ \mR^n\rightarrow\mR$ and $\psi_{\rho}:\ \mR^n\rightarrow\mR$ are defined as
$\phi_{\bb}(\bz) \triangleq \frac{1}{2}\norm{\bz-\bb}_2^2$ and $\psi_{\rho}(\bz)  \triangleq \frac{\rho}{2}\norm{\bz}_2^2$. For the simplicity of notation, the subscripts $\bb$ and $\rho$ are eliminated in the remainder of the manuscript. The following two properties are used in the subsequent analysis:
{\small
\begin{align}
\norm{\grad_{\bx^2}\phi(M\bx^2(\by))-\grad_{\bx^2}\phi(M\bx^2(\bar{\by}))}_2&=\norm{M^\top M(\bx^2(\by)-\bx^2(\bar{\by}))}_2\leq L_{\phi}\norm{M\bx^2(\by)-M\bx^2(\bar{\by})}_2, \label{eq:prop_phi}\\
\norm{\grad_{\bx}\psi(E\bx(\by))-\grad_{\bx}\psi(E\bx(\bar{\by}))}_2&=\rho\norm{E^\top E(\bx(\by)-\bx(\bar{\by}))}_2\leq L_{\psi}\norm{E\bx(\by)-E\bx(\bar{\by})}_2, \label{eq:prop_psi}
\end{align}
}
where $L_{\phi}=\norm{M^T}_2$ and $L_{\psi}=\rho\sqrt{2}$.

Given $\by\in\mR^n$, the optimization problem in \eqref{eq:augDual} can equivalently be written as
\begin{align}
\label{eq:lagrange_min_conic}
\begin{split}
\min_{\bx^1\in\mR^n,\bx^2\in\mR^n,\bs\in\mR^{|\cG|}} & \lambda\sum_{g\in\cG}s_g+\frac{1}{2}\norm{M\bx^2-\bb}_2^2+\fprod{\by,\bx^1-\bx^2}+\frac{\rho}{2}\norm{\bx^1-\bx^2}_2^2, \\
\text{s.t.} & \ \ w_g\norm{\bx^1_{j(g)}}_2\leq s_g, \qquad \forall g\in\cG.
\end{split}
\end{align}
The constraints in \eqref{eq:lagrange_min_conic} are indeed second-order cones $\cQ_g=\{(\bx_g,s_g)\in\mR^{|\cG|}\times\mR_+:\ w_g\norm{\bx_g}_2\leq s_g\}$. The KKT system for the problem \eqref{eq:lagrange_min_conic} is
\begin{subequations}
\label{eq:kkt}
\begin{align}
\by_{j(g)}-w_g\m_g+\rho(\bx^1_{j(g)}-\bx^2_{j(g)}) &= \bo \qquad \forall g\in\cG, \label{eq:kkt_1} \\
M^\top(M\bx^2-\bb)+\rho(\bx^2-\bx^1)-\by &= \bo, \label{eq:kkt_2} \\
\lambda-\nu_g &= 0 \qquad \forall g\in\cG, \label{eq:kkt_3} \\
w_g\norm{\bx^1_{j(g)}}_2 &\leq s_g \qquad \forall g\in\cG, \label{eq:kkt_4} \\
\norm{\m_g}_2 &\leq \nu_g \qquad \forall g\in\cG, \label{eq:kkt_5} \\
\m_g^\top\bx^1_{j(g)}+\frac{s_g}{w_g}\nu_g &= \bo \qquad \forall g\in\cG, \label{eq:kkt_6}
\end{align}
\end{subequations}
where $(\m_g,\nu_g)\in\mR^{|\cG|}\times\mR_+$ is the dual variable for the conic constraint. Hence, a pair $\big((\bx^1,\bx^2,\bs),(\m,\n)\big)$ is an optimal primal-dual pair if it satisfies \eqref{eq:kkt}. Using the KKT conditions, first we provide the dual error bound in Lemma~\ref{lem:dual_err_bound}. 

\begin{lemma}\label{lem:dual_err_bound}
There exists $\tau_d>0$ such that
\begin{equation}
\text{dist}(\by,Y^*)\leq \tau_d\norm{\grad g_{\rho}(\by)}_2,
\end{equation}
where $g_{\rho}(\by)$ is the augmented dual function defined in \eqref{eq:augDual}. 
\end{lemma}

\begin{proof}
Check Appendix~\ref{app:proof_lem_dual_bound}.
\end{proof}


From the proof of Lemma~\ref{lem:dual_err_bound}, we see that $\tau_d$ is indeed a xfunction of $\rho$ and $M$. Next, in Lemma~\ref{lem:boundedness} below, we show the existence of a finite saddle point to the augmented Lagrangian function and that the sequence generated by the algorithm is uniformly bounded.

\begin{lemma}\label{lem:boundedness}
Given the existence of a finite saddle point to the augmented Lagrangian function \eqref{eq:augLagrange}, for any $\rho$ and $\alpha$ such that $0<\alpha<\rho$, the sequence $\{\bx^{1,k}\}$, $\{\bx^{2,k}\}$ and $\{\by^k\}$ generated by the algorithm \eqref{eq:admm_0_1}-\eqref{eq:admm_0_3} is uniformly bounded. 
\end{lemma}

\begin{proof}
Check Appendix~\ref{app:proof_lem_boundedness}.
\end{proof}


We also need to establish the primal error bound. Unlike the dual error bound, where the gradient of the augmented $g_{\rho}(\by)$ nicely bounds the ``error", i.e., the distance of a point to the optimal solution set, the primal function is not smooth and differentiable. Quantifying the error bound for nonsmooth functions is generally performed by the \emph{proximal gradient}. For general surveys on error bounds, see \cite{pang1997,zhou2019} and references therein.

\begin{definition}\label{def:prox_grad}
Assume a convex function $f$ is decomposable as $f(\bx)=g(A\bx)+h(\bx)$, where $g$ is a strongly convex and differentiable function and $h$ is a convex (possibly nonsmooth) function, then
we can define the proximal gradient of $f$ with respect to $h$ as $$\quad \tilde{\nabla} f(\bx) :=\bx-\operatorname{prox}_{h}(\bx-\nabla(f(\bx)-h(\bx)))=\bx-\operatorname{prox}_{h}\left(\bx-A^\top \nabla g(A\bx)\right)$$
\end{definition}
If $h=0$ then the proximal gradient $\tilde{\nabla} f(\bx)$ is equal to the gradient $\nabla f(\bx)$. In general, $\tilde{\nabla} f(\bx)$ can be used as the (extended) gradient for nonsmooth minimization $\min_{\bx \in \mR^n} f(\bx) .$ For instance, we have $\tilde{\nabla} f\left(\bx^{*}\right)=\bo$
if and only if $\bx^{*}$ is a minimizer. For the Lagrangian function \eqref{eq:augLagrange}, the proximal gradient w.r.t. $\bx=(\bx^1,\bx^2)$ is defined as
{\small
\begin{equation*}
\tilde{\nabla}_{\bx} L_{\rho}(\bx^1,\bx^2;\by):=\bx-\operatorname{prox}_{\lambda\sum_{g\in\cG}w_g\norm{\bx^1_{j(g)}}_2}\left(\bx-\nabla_{\bx}(\frac{1}{2}\norm{M\bx^2-\bb}_2^2+\fprod{\by,\bx^1-\bx^2}+\frac{\rho}{2}\norm{\bx^1-\bx^2}_2^2)\right),
\end{equation*}
}
which we split into $\tilde{\nabla}\bx^1 L_{\rho}$ and $\tilde{\nabla}\bx^2 L_{\rho}$ in the proof of Lemma~\ref{lem:primal_err_bound}.
\begin{lemma}\label{lem:primal_err_bound}
Assume that $(\bx^1,\bx^2;\by)$ is in a compact set, then there exist $0<\tau_p<+\infty$ and $\delta>0$ such that
\begin{equation}\label{primal_error}
\text{dist}(\bx,\bX(\by))\leq \tau_p\norm{\tilde{\nabla}_{x} L_{\rho}(\bx^1,\bx^2;\by)}_2,
\end{equation}
for all $(\bx^1,\bx^2;\by)$ such that $\norm{\tilde{\nabla}_{x} L_{\rho}(\bx^1,\bx^2;\by)}_2\leq \delta$. Furthermore, $\tau_p$ and $\delta$ are independent of $\by$.
\end{lemma}

\begin{proof}
Check Appendix~\ref{app:proof_lem_primal_bound}.
\end{proof}

\begin{remark}
In Lemma~\ref{lem:primal_err_bound}, the condition $\norm{\tilde{\nabla}_{x} L_{\rho}(\bx^1,\bx^2;\by)}_2\leq \delta$ can be relaxed. Note that $\frac{\text{dist}(\bx,\bX(\by))}{\norm{\tilde{\nabla}_{x} L_{\rho}(\bx^1,\bx^2;\by)}_2}$ is a continuous and well-defined function for all $(\bx^1,\bx^2;\by)$ such that $\norm{\tilde{\nabla}_{x} L_{\rho}(\bx^1,\bx^2;\by)}_2$ $\geq\delta$. From the uniform boundedness of the sequence in Lemma~\ref{lem:boundedness}, it implies that there exists an upper bound $\tau$ such that $\frac{\text{dist}(\bx,\bX(\by))}{\norm{\tilde{\nabla}_{x} L_{\rho}(\bx^1,\bx^2;\by)}_2}\leq\tau$, for all $(\bx^1,\bx^2;\by)$ where $\norm{\tilde{\nabla}_{x} L_{\rho}(\bx^1,\bx^2;\by)}_2\geq\delta$. Choosing $\tau_p$ to be the maximum of $\tau_p$ (in \eqref{primal_error}) and $\tau$, we have $\text{dist}(\bx,\bX(\by))\leq \tau_p\norm{\tilde{\nabla}_{x} L_{\rho}(\bx^1,\bx^2;\by)}_2$ for all $(\bx^1,\bx^2;\by)$.
\end{remark}

\begin{theorem}\label{thm:main}
Let $\{(\bx^k,\by^k)\}$ be the sequence generated by Algorithm \ref{alg:sharing_Prox_LOG} with stepsize $\alpha\leq\frac{\rho}{2\tau_p^2\sigma^2}$, where $\rho$ is the augmented Lagrangian parameter, $\tau_p$ is the primal error bound parameter and $\sigma$ is defined in the proof. Furthermore, let $\Delta^k_p=L_\rho(\bx^{k+1}; \by^k)-g_\rho(\by^k)$ and $\Delta^k_d=g_{\rho}^{*}-g_{\rho}(\by^k)$ be the primal and dual optimality gaps at the k-th iteration, respectively. Then, we have  $$[\Delta^k_p+\Delta^k_d]\leq (\frac{1}{\lambda+1})^k[\Delta^0_p+\Delta^0_d],$$
where $\lambda=\min\{\frac{\rho-2\alpha\tau_p^2\sigma^2}{\zeta+\zeta'\tau_p^2\sigma^2},\frac{\alpha}{\tau'}\}>0$, and $\sigma$, $\zeta$, $\zeta'$, $\tau'$ are defined in the proof, which shows that the sequence $[\Delta^k_p+\Delta^k_d]$ converges to zero Q-linearly\footnote{``Q" stands for Quotient. A sequence $\{\Delta^k\}$ converges Q-linearly to $\bar{\Delta}$ for a given norm $\norm{\cdot}$ if $\norm{\Delta^{k+1}-\bar{\Delta}}/\norm{\Delta^k-\bar{\Delta}}\leq\mu$ for all $k$, for some $\mu\in(0,1)$.}. 
\end{theorem}

\begin{proof}
Check Appendix~\ref{app:main_thm}.
\end{proof}

\begin{remark}\label{rem:rate_1}
Theorem~\ref{thm:main} proves that with the stepsize small enough such that  $\alpha\leq0.5\rho\tau_p^{-2}\sigma^{-2}$, we have $\Delta^k_p+\Delta^k_d\leq\epsilon$ after $k\geq(\log(1+\lambda))^{-1}\log(\frac{\Delta^0_p+\Delta^0_d}{\epsilon})$ iterates, where $\lambda$ is defined in the statement of the theorem, i.e., an $\epsilon$-optimal solution is obtained in $k\geq O(\log(1/\epsilon))$ iterates.
\end{remark}

\begin{remark}\label{rem:rate_2}
If the stepsize $\alpha$ is small enough such that $\alpha\leq0.5\rho\tau_p^{-2}\sigma^{-2}$ and $\frac{\rho-2\alpha\tau_p^2\sigma^2}{\zeta+\zeta'\tau_p^2\sigma^2} > \frac{\alpha}{\tau'}$, then $\lambda=\alpha/\tau'=\alpha\rho/\tau_d^2$. Hence, the $O(1)$ constant in the convergence rate of the algorithm would be $(\log(1+\alpha\rho/\tau_d^2))^{-1}$. Note that this scenario happens if the stepsize $\alpha$ is small enough. In this scenario, since $\tau_d=\max\{\norm{M^\top}^2,2\rho\}$ (see the proof of Theorem~\ref{thm:main}), if $\rho<(1/2)\norm{M^\top}^2$, then the $O(1)$ constant is $(\log(1+\alpha\rho/\norm{M^\top}^4))^{-1}$ which is an increasing function of $\norm{M^\top}$. Furthermore, since $\norm{M}_F/d\leq\norm{M^\top}$ and given the binary structure of $M$, $\norm{M}_F=m$ where $m$ is the number of nonzero elements of $M$. Hence, the $O(1)$ constant is larger (i.e, worst-case convergence is slower but still linear) when $m$ is larger. This mainly happens for longer DAGs compared to wider ones (given equal number of nodes), i.e., those with more ancestry structures. Furthermore, in the same scenario, smaller $\rho$ makes the constant larger, and the convergence slower. Otherwise, if $\rho\geq(1/2)\norm{M^\top}^2$, then the $O(1)$ constant is $(\log(1+\alpha/(4\rho)))^{-1}$ which gets larger, i.e. convergence is slower, for larger values of $\rho$. Finally, in both scenarios, bigger stepsize $\alpha$ (up to the linear rate upper bound of $0.5\rho\tau_p^{-2}\sigma^{-2}$) makes the constant smaller and, the worst-case convergence faster. 
\end{remark}

\section{Numerical experiments}\label{sec:numerical}
\vspace{0.2cm}
\subsection{Simulation studies} \label{sec:simulations}
This section provides our numerical studies on the performance of the proposed algorithm to evaluate the proximal operator of the LOG penalty. We compare the convergence rate of the proposed ADMM algorithm with the sharing scheme, i.e. Algorithm~\ref{alg:sharing_Prox_LOG}, with five other other algorithms including the Cyclic Block Coordinate Descent (C-BCD), shown in Algorithm~\ref{alg:BCD_LOG}, and its randomized version (R-BCD) (see \cite{richtarik2014iteration}), Proximal Gradient Descent (PGM) with backtracking (which is the ISTA algorithm in \cite{beck2009fast}), Accelerated PGM (ACC-PGM) with backtracking (which is the FISTA algorithm in \cite{beck2009fast}), and Hierarchical Sparse Modeling (HSM) by \citet{yan2017hierarchical} to find the proximal mapping of the LOG penalty for the six DAGs shown in Figure~\ref{fig:DAGs} on simulated data.

From the six different graphs shown in Figure~\ref{fig:DAGs}, four DAGs are indeed tree graphs with different structures with 101, 101, 127, 201 nodes (DAGs (a)-(c) and DAG (e)); one reverse binary tree (DAG (d)), and one random DAG with 100 nodes and 98 edges (DAG (e)) are also considered in the study. Note that each node represents a single parameter, i.e. $d=N$.

In each simulation $\bb \in \mathbb{R}^d$ is sampled from $\cN(\bo,\bI_d)$. Considering problem \eqref{eq:main_2}, the parameters are set as $\lambda=0.1$ and $w_g=|g|^{1/2}$. The step sizes of the PGM and ACC-PGM methods are selected by backtracking. The $\rho$ parameter in the ADMM algorithm is set to be a number between 1 and 20 and $\alpha$ is set to 1. All simulations in this section are run on a laptop with 2.4 GHz Intel Core i9 CPU and 32 GB memory using only one thread. The simulation is replicated 10 times and the convergence plots are obtained over their averages. Figures~\ref{fig:converge_simulation} and \ref{fig:con} show the relative error of the objective function versus iteration and time, respectively. The optimal objective function value $f^*$ is taken to be the minimum objective function value of the converging solution over the six different methods. Note that the C-BCD and R-BCD methods are not included in Figure~\ref{fig:converge_simulation}, as the notion of iteration for block coordinate methods is different and not comparable with the other methods. All of the implementations and the corresponding codes are available at \url{https://github.com/samdavanloo/ProxLOG}.

As we can see in Figure~\ref{fig:converge_simulation}, Algorithm~\ref{alg:sharing_Prox_LOG} shows linear convergence for all graphs which matches our theoretical upper bound in Section~\ref{sec:rate}. We want to reiterate that the objective function of the LOG penalty, i.e. \eqref{eq:main_2}, is \emph{not} strongly convex, but the algorithm still converges linearly. The HSM algorithm converges in finite-steps for DAG (b) which matches the theory proposed in their paper \citep{yan2017hierarchical}, given the path structure of this graph. However, HSM's convergence becomes slower as the graph grows in width, e.g., DAG (a), where ADMM and ACC-PGM are the fastest methods. As expected, ACC-PGM converges faster compared to PGM for all graphs. We also note that ADMM achieves the best objective function value $f^*$ in most of the experiments for all of the graphs.


With respect to the time, as shown in Figure~\ref{fig:con}, we compare the convergence speed of all six methods for the six graphs shown in Figure~\ref{fig:DAGs}. R-BCD algorithm has the fastest convergence for DAGs (a) and (c), which are instances of two wide graphs, and ADMM and ACC-PGM are at the second place. For DAG (b), which is a tree with two long path graphs, after HSM which provably converges in finite steps, C-BCD and ADMM are at the second place. In the asymmetric DAG (e), C-BCD decreases the function values fast at the beginning; however, its convergence becomes slower (maybe sublinear), while ADMM continues its linear convergence trend. In the random DAG (e), C-BCD is the fastest algorithm and ADMM is at the second place. We should note that BCD algorithms cannot be parallelized given the sequential nature of their computations; hence, they cannot benefit from parallel computing resources. In general, across different graph topologies, ADMM has the most robust performance from the convergence speed perspective; furthermore, in many instances, it produces the best final (minimum) objective function value.

\begin{figure}[htbp]
    \centering
    \begin{adjustbox}{minipage=\linewidth,scale=0.8}
    \begin{subfigure}[b]{0.4\textwidth}
        \centering
        \includegraphics[width=1\textwidth]{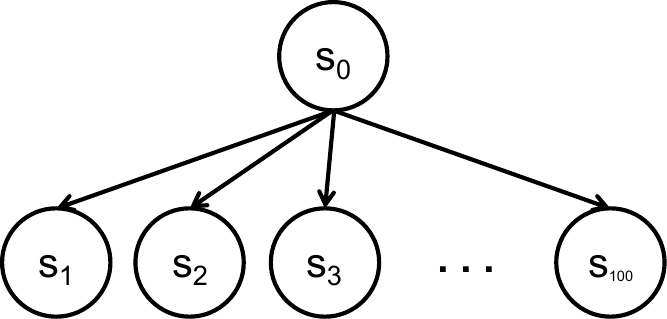}
        \mbox{} \\
        \mbox{}
        \caption{Two-layer tree, $d=101$}
         \label{fig:1to100}
    \end{subfigure}
    \hfill
    \begin{subfigure}[b]{0.4\textwidth}
         \centering
         \includegraphics[width=0.55\textwidth]{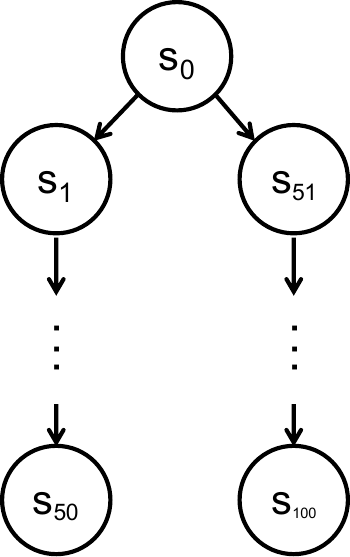}
         \caption{One root two paths tree, $d=101$}
         \label{fig:two_path}
     \end{subfigure}    
      \hfill
    \begin{subfigure}[b]{0.4\textwidth}
        \centering
        \includegraphics[width=1\textwidth]{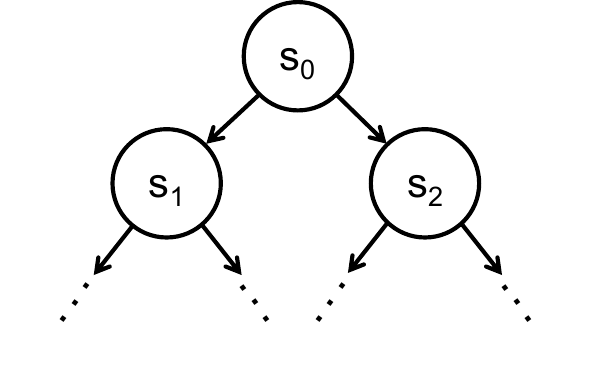}
         \caption{Binary tree, $d=127$}
         \label{fig:binary}
    \end{subfigure}
    \hfill
    \begin{subfigure}[b]{0.4\textwidth}
        \centering
        \includegraphics[width=1\textwidth]{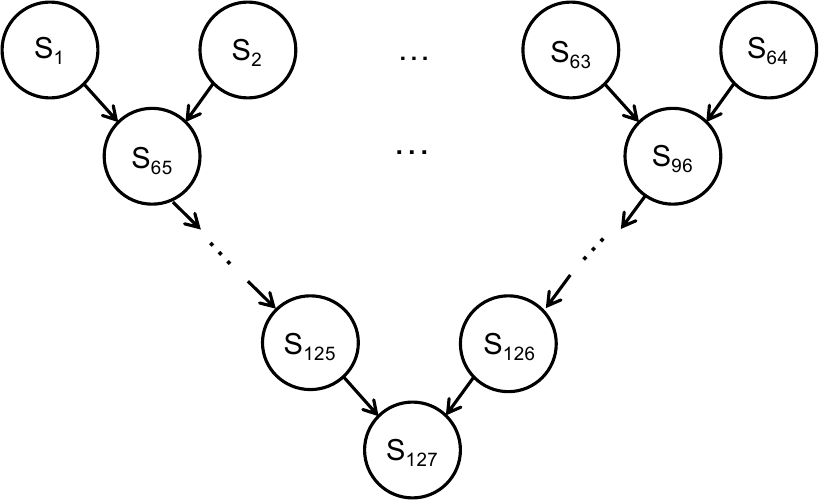}
         \caption{Reverse binary tree, $d=127$}
         \label{fig:binary}
    \end{subfigure}
    \hfill
          \begin{subfigure}[b]{0.4\textwidth}
         \centering
         \includegraphics[width=1\textwidth]{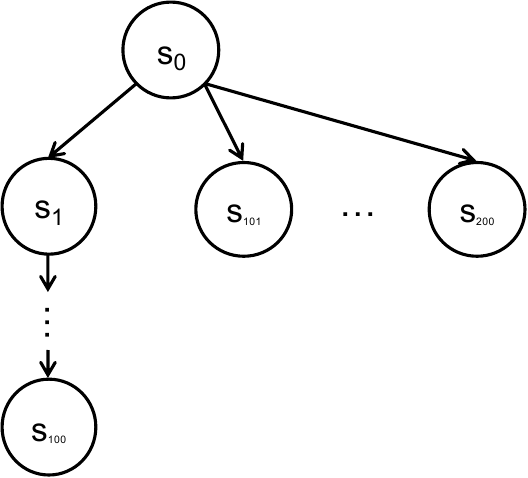}
         \mbox{}
         \caption{Asymmetric tree, $d=201$}
         \label{fig:asym}
     \end{subfigure}
     \hfill
       \begin{subfigure}[b]{0.4\textwidth}
         \centering
    	 \includegraphics[width=1.2\textwidth]{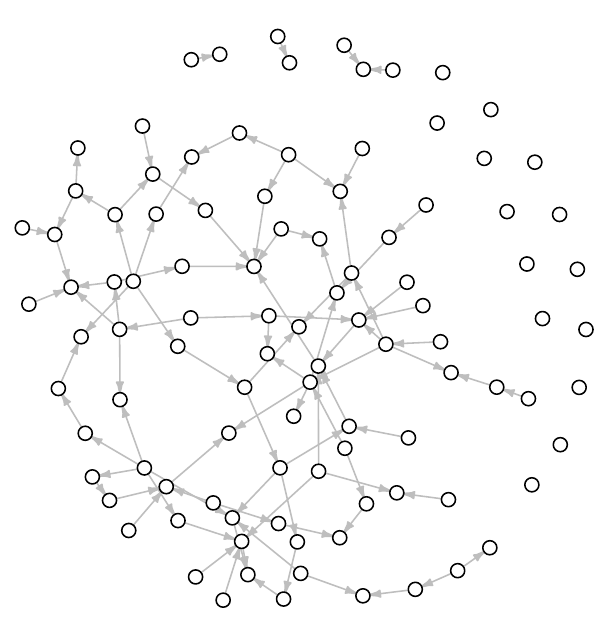}
         \caption{Random DAG, $d=100$}
         \label{fig:DAG40}
     \end{subfigure}
     \end{adjustbox}
        \caption{Six different DAGs considered in the simulation study}
        \label{fig:DAGs}
\end{figure}

\begin{figure}[htbp]
    \centering
    \begin{adjustbox}{minipage=\linewidth,scale=0.8}
    \begin{subfigure}[b]{0.45\textwidth}
        \centering
        \includegraphics[width=1\textwidth]{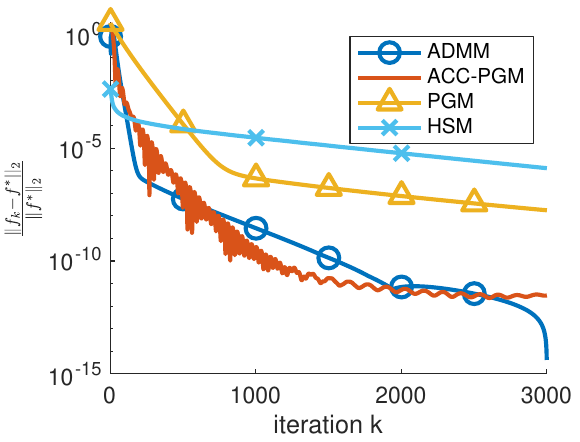}
        \caption{Two-layer tree, $d=101$}
         \label{fig:1to100_res}
    \end{subfigure}
    \hfill
    \begin{subfigure}[b]{0.45\textwidth}
         \centering
         \includegraphics[width=1\textwidth]{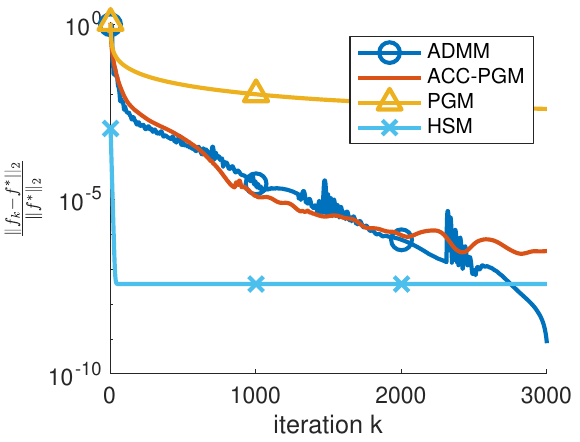}
         \caption{One root two paths tree, $d=101$}
         \label{fig:twopath_res}
     \end{subfigure}
      \hfill
    \begin{subfigure}[b]{0.45\textwidth}
        \centering
        \includegraphics[width=1\textwidth]{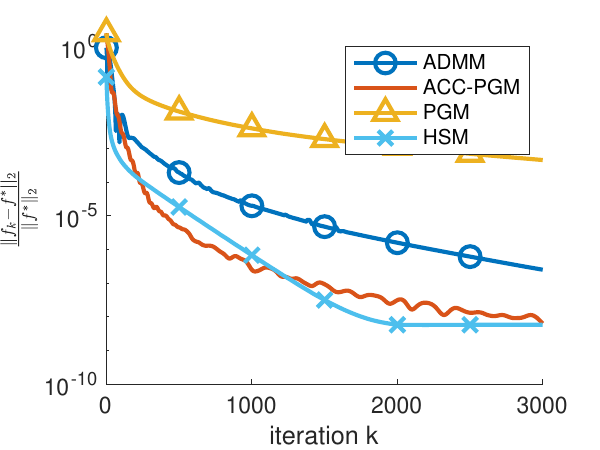}
         \caption{Binary tree, $d=127$}
         \label{fig:binary_res}
    \end{subfigure}
     \hfill
    \begin{subfigure}[b]{0.45\textwidth}
        \centering
        \includegraphics[width=1\textwidth]{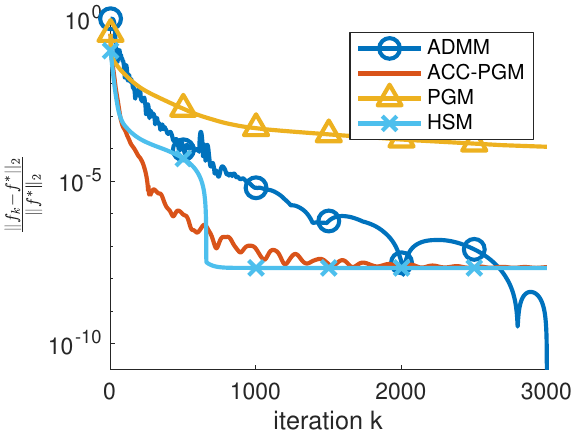}
         \caption{Reverse binary tree, $d=127$}
         \label{fig:rev_binary_res}
    \end{subfigure}
     \hfill
      \begin{subfigure}[b]{0.45\textwidth}
         \centering
         \includegraphics[width=1\textwidth]{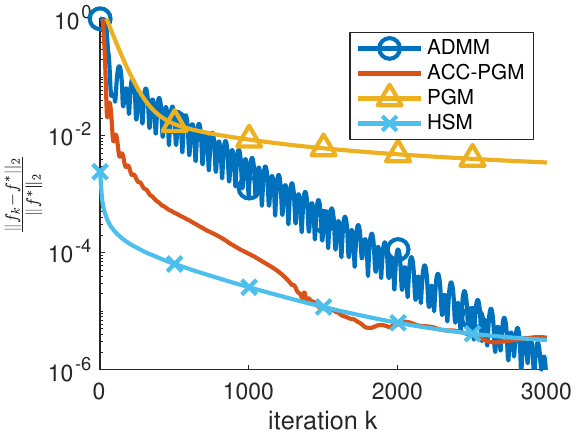}
         \caption{Asymmetric tree, $d=201$}
         \label{fig:asym_res}
     \end{subfigure}
    \hfill
     \centering
      \begin{subfigure}[b]{0.45\textwidth}
         \centering
         \includegraphics[width=1\textwidth]{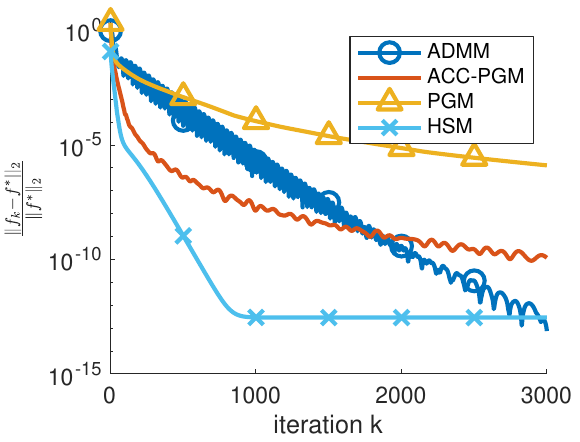}
         \caption{Random DAG, $d=100$}
         \label{fig:randomDAG100_res}
     \end{subfigure}
     
     \end{adjustbox}
        \caption{Convergence of the proposed ADMM, Proximal Gradient Method (PGM), and Accelerated PGM (ACC-PGM), Hierarchical Sparse Modeling (HSM) algorithms versus iteration for the six DAGs in Figure~\ref{fig:DAGs}.}
        \label{fig:converge_simulation}    
\end{figure}

\begin{figure}[htbp]
    \centering
    \begin{adjustbox}{minipage=\linewidth,scale=0.8}
    \begin{subfigure}[b]{0.45\textwidth}
        \centering
        \includegraphics[width=1\textwidth]{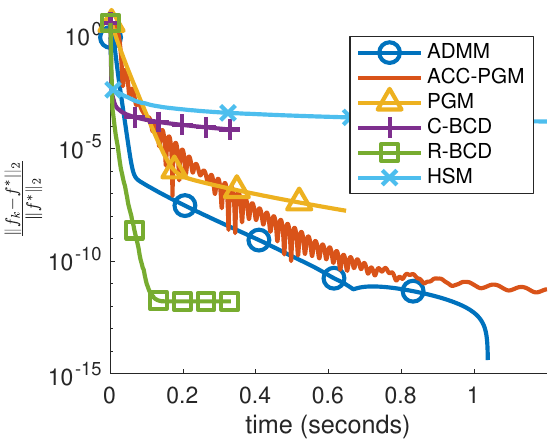}
        \caption{Two-layer tree, $d=101$}
         \label{fig:1to100_res}
    \end{subfigure}
   \hfill
    \begin{subfigure}[b]{0.45\textwidth}
         \centering
         \includegraphics[width=1\textwidth]{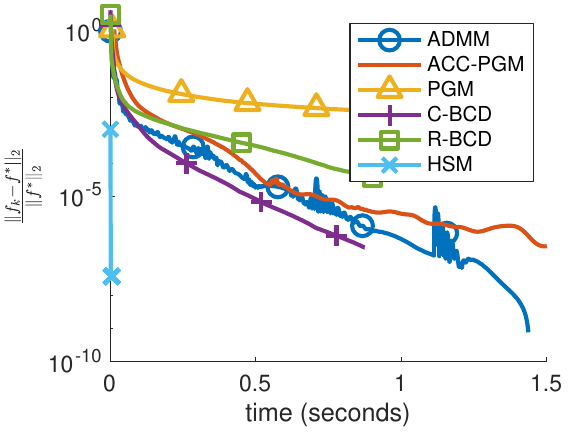}
         \caption{One root two paths tree, $d=101$}
         \label{fig:twopath_res}
     \end{subfigure}
      \hfill
    \begin{subfigure}[b]{0.45\textwidth}
        \centering
        \includegraphics[width=1\textwidth]{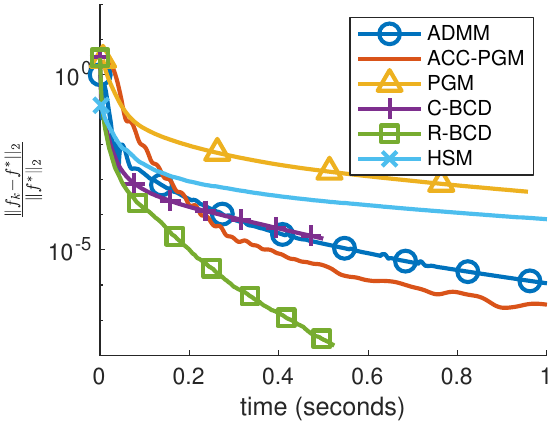}
         \caption{Binary tree, $d=127$}
         \label{fig:binary_res}
    \end{subfigure}
     \hfill
    \begin{subfigure}[b]{0.45\textwidth}
        \centering
        \includegraphics[width=1\textwidth]{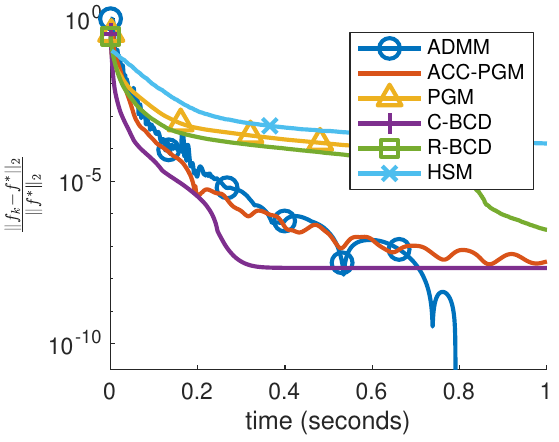}
         \caption{Reverse binary tree, $d=127$}
         \label{fig:rev_binary_res}
    \end{subfigure}
      \hfill
      \begin{subfigure}[b]{0.45\textwidth}
         \centering
         \includegraphics[width=1\textwidth]{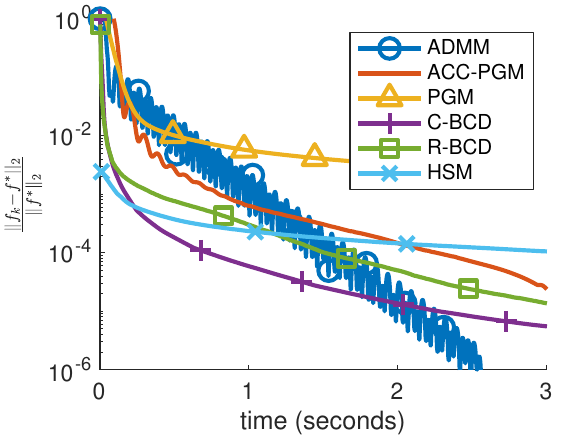}
         \caption{Asymmetric tree, $d=201$}
         \label{fig:asym_res}
     \end{subfigure}
     \hfill
     \centering
      \begin{subfigure}[b]{0.45\textwidth}
         \centering
         \includegraphics[width=1\textwidth]{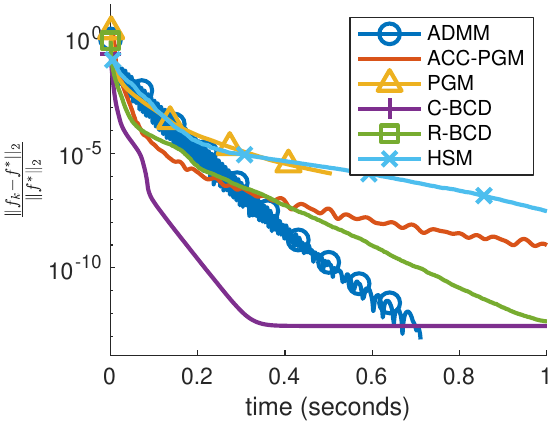}
         \caption{Random DAG, $d=100$}
         \label{fig:randomDAG100_res}
     \end{subfigure}
     \end{adjustbox}
        \caption{Convergence of the proposed ADMM, Proximal Gradient Method (PGM), Accelerated PGM (ACC-PGM), Cyclic BCD (C-BCD), Randomized BCD (R-BCD), and Hierarchical Sparse Modeling (HSM) algorithms over time for the six DAGs shown in Figure~\ref{fig:DAGs}.}
        \label{fig:con}  
\end{figure}

\subsection{Two Applications}\label{sec:app}
The proposed algorithm allows efficient evaluation of the proximal operator of the LOG penalty. Evaluating the proximal operator is generally needed iteratively within a master optimization algorithm that tries to solve an underlying statistical learning problem. Note that the master problem might be a convex or nonconvex optimization problem. To demonstrate practicality of the proposed algorithm, in this section, we consider two statistical learning problems on topic modeling and classification. The topic modeling application is a dictionary learning problem for NeurIPS proceedings. The second application relates to a breast cancer classification problem using gene expression data. 
 
\subsubsection{Topic modeling of NeurIPS proceedings.}\label{sec:topic}
We are interested in solving the topic modeling problem represented as the dictionary learning problem \eqref{eq:dic_learn} penalized with the LOG penalty. Introduction of the LOG penalty is to force the resulting topics to form a tree structure~\citep{jenatton2011proximal}. The underlying statistical learning problem can be written as
\begin{equation}\label{eq:dic_learn}
	\min_{\bD \in D_1^+, \bA \in \mathbb{R}_+^{k \times n}}\sum_{j=1}^{n} \left[ \frac{1}{2}\Vert \bx^j - \bD \alpha^j \Vert _2^2 + \lambda \Omega_{\text{LOG}}(\alpha^j) \right]
\end{equation}
where $\bX = [\bx^1, \bx^2, \cdots , \bx^n] \in \mathbb{R}^{m\times n}$ represents frequencies of $m$ words in $n$ articles and the $i$-th element of $\bx^j$ is the frequency of the $i$-th word in the $j$-th article. $\bD = [\bd^1,\bd^2,\cdots, \bd^k] \in  D_1^+$ is the dictionary of $k$ topics to be learnt where $D_1^+\triangleq \{\bD \in \mathbb{R}_+^{m \times k}:~\norm{\bd^j}_1 \leq 1, j =1,2,\cdots, k\}$.
 Furthermore, $\bA\triangleq [\alpha^1, \alpha^2, \cdots, \alpha^n] \in \mathbb{R}^{k\times n}_+ $ is the corresponding coefficients for each article such that $\bx^j \approx \bD \alpha^j$. 

Following the framework of \cite{jenatton2011proximal}, we solve \eqref{eq:dic_learn} using an alternating minimization scheme, i.e., updating $\bD$ and $\bA$ one at a time while keeping the other one fixed. The $\bD$ update is performed using  C-BCD algorithm, taking its columns as the blocks, using the algorithm of \cite{mairal2010online}. The $\bA$ update is performed by the accelerated proximal gradient method ACC-PGM~\citep{beck2009fast}. To evaluate the proximal operator of the LOG penalty, we implemented the proposed ADMM (with and without parallelization), R-BCD, and C-BCD algorithms. Given that the number of groups for this application is $|\cG|=13$, the first block update of the (parallel) ADMM (lines 4 and 5 of Algorithm~\ref{alg:sharing_Prox_LOG}) for each $\alpha^j$ is parallelized over 13 processing nodes. We also included unparallelized ADMM algorithm for comparison. Note that BCD algorithms cannot be parallelized.

These three nested algorithms are implemented for the NeurIPS proceedings from 1996 to 2015~\citep{perrone2016poisson}. The dataset contains $n=1846$ articles with $m = 11463$ words that excludes stop words and words occurring less than 50 times. We set $k=13$, $\lambda = 2^{-15}$, and followed the hierarchical structure proposed by \cite{jenatton2011proximal} to induce a tree of topics - see Figure~\ref{fig:topic}. The experiment is run on a cluster with 2.4GHz CPU and 128GB memory using 28 threads. Note that the columns of the $\bA$ matrix, i.e. $\alpha^j$, can be updated in parallel over $n=1846$ articles for all three methods.

Figure \ref{fig:topic_convergence} shows the convergence behavior of the algorithms discussed above using the norm of the proximal gradient (see Definition~\ref{def:prox_grad}) and $(1/n)\norm{A^k-A^{k-1}}_F+(1/m)\norm{D^k-D^{k-1}}_F$ as two convergence measures. Evaluating the proximal operator of the columns of the $A$ matrix using the C-BCD and parallelized ADMM are the two fastest methods, but the quality of the C-BCD solution seems to be better. Even though the number of groups is relatively small $|\cG|=13$, parallelization is significantly reducing the convergence time. Reduction of the convergence time by parallelization of ADMM will even be more significant when the number of groups is bigger -- see the application in Section~\ref{sec:breast_cancer}.
\begin{figure}[htbp]
	\vspace{-2cm}
        \centering
        \includegraphics[width=0.8\textwidth]{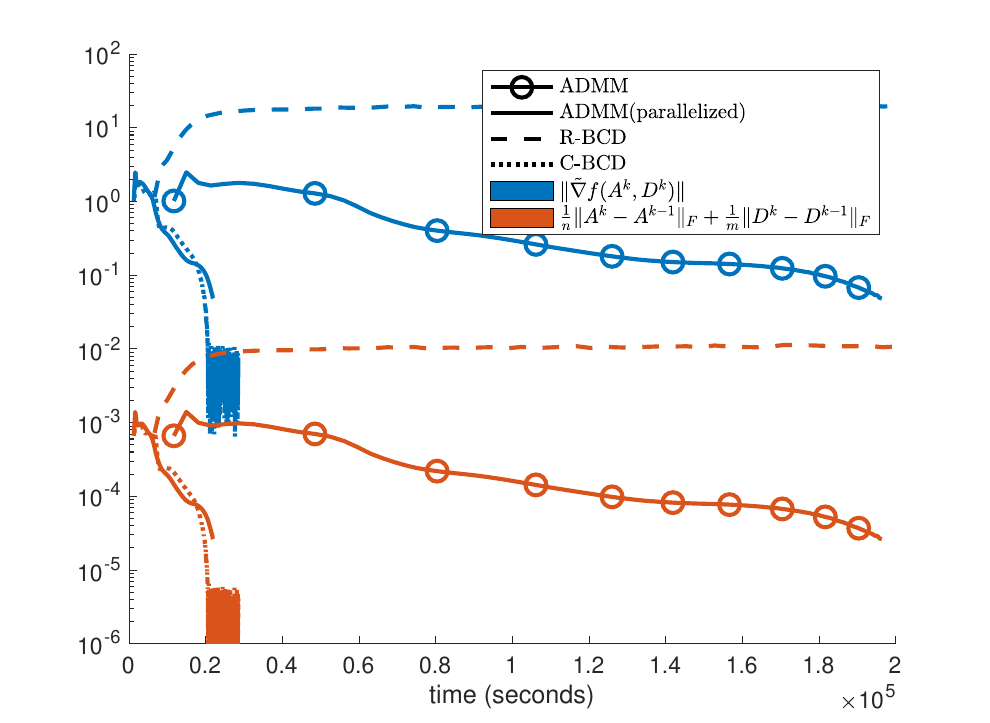}
    	\caption{Convergence of the algorithm for the topic modeling application on two different convergence measures. The $A$-update is performed by ACC-PGM algorithm where its proximal operator is evaluated by the ADMM (with/without parallelization), R-BCD, and C-BCD methods.}
   	 \label{fig:topic_convergence}
\end{figure}
Figure~\ref{fig:topic} depicts the learnt hierarchal topics with the 7 most frequent words. The root is a general topic while the leafs are more specific and narrower topics.
\begin{figure}[h]
  \centering
  \includegraphics[width=0.8\linewidth]{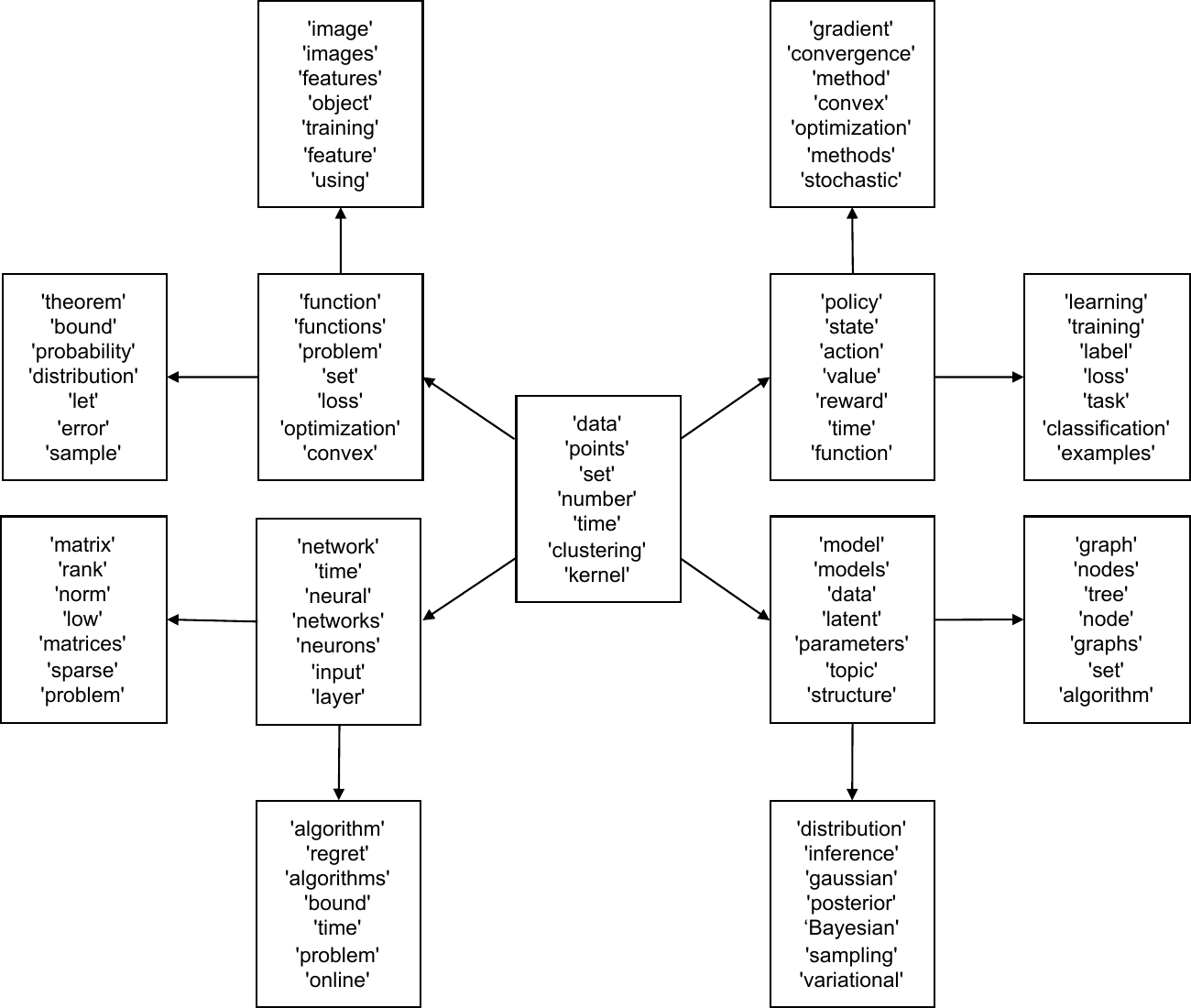}
  \caption{Hierarchical topics of NeurIPS from 1996 to 2015}
  \label{fig:topic}
\end{figure}

\subsubsection{Breast cancer classification.}\label{sec:breast_cancer}
This section discusses fitting a logistic regression model penalized with the LOG penalty to classify breast cancer  based on gene expression levels. It is known that genes functionalities are highly affected by their two-way interactions which might be a priori known based on a protein-protein network. Hence, to identify contributing genes for cancer metastasis, it is important to consider such structures.

We use the breast cancer dataset of \cite{van2002gene} that consists of 8141 gene expression data for 78 metastatic and 217 non-metastatic patients. Following the experimental settings in \cite{obozinski2011group}, we build groups of genes based on the protein-protein network of \cite{chuang2007network}. Every two genes connected directly by an edge in the network are assigned as a group. The total number of groups for this application is $|\cG|=522$. Given that groups have overlaps on many nodes, LOG penalty is used to capture the relationship within groups. The genes that are not contained in the network are eliminated and the 500 most correlated genes are selected.

The learning problem involves minimizing the logistic loss function regularized with the LOG penalty that can be written as
\begin{equation} \label{eq:logistic_LOG}
 \min_{\th} f(\th)=-\frac{1}{m}\sum_{i=1}^{m} \{y^{(i)} \log h_\theta (\bx^{(i)}) +(1-y^{(i)})\log(1-h_\theta(\bx^{(i)}))\} +\lambda \Omega_{\text{LOG}}(\th),
\end{equation}
where $(\bx^{(i)}, y^{(i)})$ is the input data, $y^{(i)}$ is either 0 or 1, and $h_\theta(\bx)\triangleq \frac{1}{1+e^{-\th^T \bx}}$. The underlying learning problem is solved by ACC-PGM algorithm while the proximal operator of the LOG penalty is evaluated by the proposed ADMM (with/without parallelization), R-BCD, and C-BCD algorithms. The $\lambda$ parameter is set equal to $10^{-3}$. The experiment is run on a cluster with 2.4GHz CPU and 128GB memory using 28 threads.

Validation of the classification performance with the LOG penalty for such a problem is performed e.g. in \cite{obozinski2011group}; so, we only focus on the convergence behavior of the proposed algorithm. The left plot in Figure~\ref{fig:cancer} shows the convergence of ACC-PGM with different proximal evaluators. While the ADMM method without parallelization is faster than R-BCD and C-BCD methods, parallelization of its first block over the available 28 nodes significantly decrease its convergence time. Note that such parallelization theoretically reduces the time further up to 522 nodes which is the number of groups $|\cG|$.


We also examine the effect of the LOG penalty for gene selection. For visual convenience,  we increase $\lambda$ to from $0.001$ to $0.05$ to make the regression coefficients sparser and evaluate the relationships of selected and unselected genes. The right plot in Figure~\ref{fig:cancer} is a subset of the network of 500 genes. Each node represents a gene and the edges are known a priori from the protein-protein network. Nonzero coefficients in the final model identify genes which are correlated with breast cancer metastasis. From this result, it is clear that connected genes are prone to be selected simultaneously which supports the rationality of the LOG penalty for this application.

\begin{figure}[htbp]
    \centering
         \includegraphics[width=0.48\textwidth]{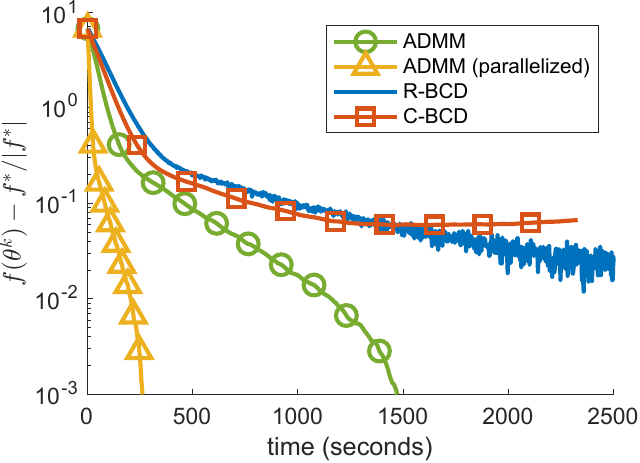}
    \hfill
         \includegraphics[width=0.50\linewidth]{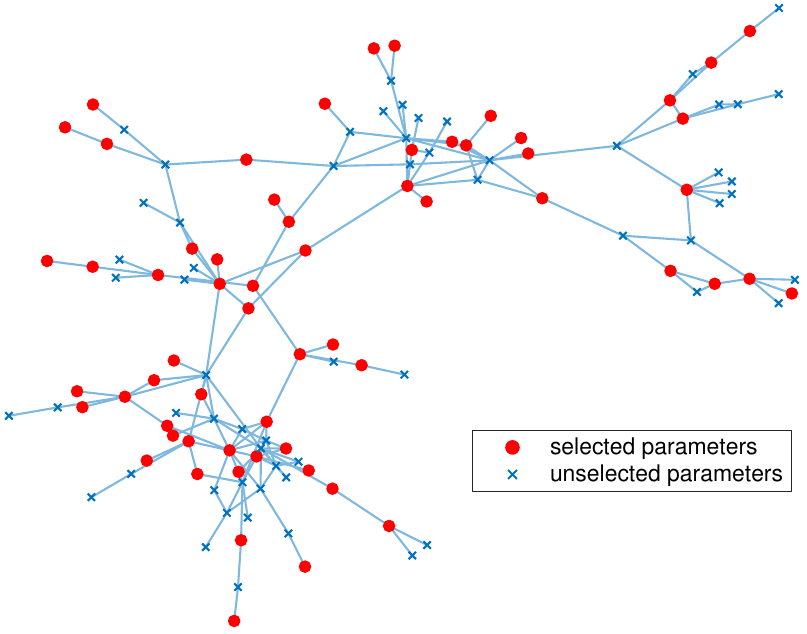}
    \caption{Breast cancer classification: \textbf{(Left)} Convergence of the ACC-PGM method where the proximal map is evaluated using ADMM (with/without parallelization), R-BCD, and C-BCD methods.\textbf{(Right)} Part of the gene network: red nodes are nonzero genes while blue nodes are zero ones. Edge structures are known a priori from the protein-protein network - see~\cite{chuang2007network}.}
        \label{fig:cancer}
\end{figure}

\section{Concluding remarks}\label{sec:conclusion}
The paper discusses an efficient algorithm to find the proximal mapping of the Latent Overlapping Group (LOG) lasso penalty to induce hierarchical sparsity structure represented by any general DAG. The sharing scheme for the underlying ADMM algorithm allows maximum parallelization over (potentially) many number of groups which allows solving large-scale instances of the underlying optimization problems which could be convex or nonconvex loss function. On the theoretical side, the paper establishes global linear rate of convergence in the \emph{absence of strong convexity}. The rate analysis is performed through the elegant error bound theory. Furthermore, the paper investigates the effect of graph structure on the speed of convergence of the algorithm. The numerical results confirms our theoretical convergence rate for different directed acyclic graphs with different sizes.

\bibliographystyle{authordate1} 
\bibliography{refs}

\begin{thebibliography}{}

\bibitem[\protect\citename{Atamt{\"u}rk \& G{\'o}mez,
  }2018]{atamturk2018strong}
Atamt{\"u}rk, Alper, \& G{\'o}mez, Andr{\'e}s. 2018.
\newblock Strong formulations for quadratic optimization with M-matrices and
  indicator variables.
\newblock {\em Mathematical Programming}, {\bf 170}(1), 141--176.

\bibitem[\protect\citename{Bach {\em et~al.}, }2012]{bach2012structured}
Bach, Francis, Jenatton, Rodolphe, Mairal, Julien, Obozinski, Guillaume, {\em
  et~al.} 2012.
\newblock Structured sparsity through convex optimization.
\newblock {\em Statistical Science}, {\bf 27}(4), 450--468.

\bibitem[\protect\citename{Bach {\em et~al.}, }2013]{bach2013learning}
Bach, Francis, {\em et~al.} 2013.
\newblock Learning with submodular functions: A convex optimization
  perspective.
\newblock {\em Foundations and Trends{\textregistered} in Machine Learning},
  {\bf 6}(2-3), 145--373.

\bibitem[\protect\citename{Bach, }2010]{bach2010structured}
Bach, Francis~R. 2010.
\newblock Structured sparsity-inducing norms through submodular functions.
\newblock {\em Pages  118--126 of:} {\em Advances in Neural Information
  Processing Systems}.

\bibitem[\protect\citename{Beck \& Teboulle, }2009]{beck2009fast}
Beck, Amir, \& Teboulle, Marc. 2009.
\newblock A fast iterative shrinkage-thresholding algorithm for linear inverse
  problems.
\newblock {\em SIAM journal on imaging sciences}, {\bf 2}(1), 183--202.

\bibitem[\protect\citename{Bertsekas, }1999]{bertsekas1999nonlinear}
Bertsekas, Dimitri~P. 1999.
\newblock {\em Nonlinear programming}.
\newblock Athena scientific Belmont.

\bibitem[\protect\citename{Bertsimas \& Van~Parys, }2017]{bertsimas2017sparse}
Bertsimas, Dimitris, \& Van~Parys, Bart. 2017.
\newblock Sparse high-dimensional regression: Exact scalable algorithms and
  phase transitions.
\newblock {\em arXiv preprint arXiv:1709.10029}.

\bibitem[\protect\citename{Bertsimas {\em et~al.}, }2016]{bertsimas2016best}
Bertsimas, Dimitris, King, Angela, Mazumder, Rahul, {\em et~al.} 2016.
\newblock Best subset selection via a modern optimization lens.
\newblock {\em The Annals of Statistics}, {\bf 44}(2), 813--852.

\bibitem[\protect\citename{Bertsimas {\em et~al.}, }2019]{bertsimas2019unified}
Bertsimas, Dimitris, Cory-Wright, Ryan, \& Pauphilet, Jean. 2019.
\newblock A unified approach to mixed-integer optimization: Nonlinear
  formulations and scalable algorithms.
\newblock {\em arXiv preprint arXiv:1907.02109}.

\bibitem[\protect\citename{Bien {\em et~al.}, }2013]{bien2013lasso}
Bien, Jacob, Taylor, Jonathan, \& Tibshirani, Robert. 2013.
\newblock A lasso for hierarchical interactions.
\newblock {\em Annals of statistics}, {\bf 41}(3), 1111.

\bibitem[\protect\citename{Boyd {\em et~al.}, }2011]{boyd2011distributed}
Boyd, Stephen, Parikh, Neal, Chu, Eric, Peleato, Borja, \& Eckstein, Jonathan.
  2011.
\newblock Distributed optimization and statistical learning via the alternating
  direction method of multipliers.
\newblock {\em Foundations and Trends{\textregistered} in Machine Learning},
  {\bf 3}(1), 1--122.

\bibitem[\protect\citename{Chouldechova \& Hastie,
  }2015]{chouldechova2015generalized}
Chouldechova, Alexandra, \& Hastie, Trevor. 2015.
\newblock Generalized additive model selection.
\newblock {\em arXiv preprint arXiv:1506.03850}.

\bibitem[\protect\citename{Chuang {\em et~al.}, }2007]{chuang2007network}
Chuang, Han-Yu, Lee, Eunjung, Liu, Yu-Tsueng, Lee, Doheon, \& Ideker, Trey.
  2007.
\newblock Network-based classification of breast cancer metastasis.
\newblock {\em Molecular systems biology}, {\bf 3}(1), 140.

\bibitem[\protect\citename{Eckstein \& Yao, }2012]{Eckstein2012ADMM}
Eckstein, Jonathan, \& Yao, Wang. 2012.
\newblock Augmented Lagrangian and Alternating Direction Methods for Convex
  Optimization: A Tutorial and Some Illustrative Computational Results.
\newblock {\em RUTCOR Research Reports}, {\bf 32}(Suppl.~3).

\bibitem[\protect\citename{Glowinski, }1984]{glowinski}
Glowinski, Roland. 1984.
\newblock {\em Numerical methods for nonlinear variational problems}.
\newblock Springer.

\bibitem[\protect\citename{Haris {\em et~al.}, }2016]{haris2016convex}
Haris, Asad, Witten, Daniela, \& Simon, Noah. 2016.
\newblock Convex modeling of interactions with strong heredity.
\newblock {\em Journal of Computational and Graphical Statistics}, {\bf 25}(4),
  981--1004.

\bibitem[\protect\citename{Hastie {\em et~al.}, }2015]{hastie15statistical}
Hastie, Trevor, Tibshirani, Robert, \& Wainwright, Martin. 2015.
\newblock {\em Statistical learning with sparsity}.
\newblock CRC press.

\bibitem[\protect\citename{Hazimeh \& Mazumder, }2019]{hazimeh2019learning}
Hazimeh, Hussein, \& Mazumder, Rahul. 2019.
\newblock Learning Hierarchical Interactions at Scale: A Convex Optimization
  Approach.
\newblock {\em arXiv preprint arXiv:1902.01542}.

\bibitem[\protect\citename{Hoffman, }1952]{hoffman1952}
Hoffman, Alan~J. 1952.
\newblock On approximate solutions of systems of linear inequalities.
\newblock {\em Journal of Research of the National Bureau of Standards}, {\bf
  49}(4), 263--265.

\bibitem[\protect\citename{Hong \& Luo, }2017]{hong2017linear}
Hong, Mingyi, \& Luo, Zhi-Quan. 2017.
\newblock On the linear convergence of the alternating direction method of
  multipliers.
\newblock {\em Mathematical Programming}, {\bf 162}(1-2), 165--199.

\bibitem[\protect\citename{Jacob {\em et~al.}, }2009]{jacob2009group}
Jacob, Laurent, Obozinski, Guillaume, \& Vert, Jean-Philippe. 2009.
\newblock Group lasso with overlap and graph lasso.
\newblock {\em Pages  433--440 of:} {\em Proceedings of the 26th annual
  international conference on machine learning}.
\newblock ACM.

\bibitem[\protect\citename{Jenatton {\em et~al.}, }2011a]{jenatton2011proximal}
Jenatton, Rodolphe, Mairal, Julien, Obozinski, Guillaume, \& Bach, Francis.
  2011a.
\newblock Proximal methods for hierarchical sparse coding.
\newblock {\em Journal of Machine Learning Research}, {\bf 12}(Jul),
  2297--2334.

\bibitem[\protect\citename{Jenatton {\em et~al.},
  }2011b]{jenatton2011structured}
Jenatton, Rodolphe, Audibert, Jean-Yves, \& Bach, Francis. 2011b.
\newblock Structured variable selection with sparsity-inducing norms.
\newblock {\em Journal of Machine Learning Research}, {\bf 12}(Oct),
  2777--2824.

\bibitem[\protect\citename{Lim \& Hastie, }2015]{lim2015learning}
Lim, Michael, \& Hastie, Trevor. 2015.
\newblock Learning interactions via hierarchical group-lasso regularization.
\newblock {\em Journal of Computational and Graphical Statistics}, {\bf 24}(3),
  627--654.

\bibitem[\protect\citename{Luo \& Tseng, }1993]{luo1993convergence}
Luo, Zhi-quan, \& Tseng, Paul. 1993.
\newblock On the convergence rate of dual ascent methods for linearly
  constrained convex minimization.
\newblock {\em Mathematics of Operations Research}, {\bf 18}(4), 846--867.

\bibitem[\protect\citename{Mairal {\em et~al.}, }2010]{mairal2010online}
Mairal, Julien, Bach, Francis, Ponce, Jean, \& Sapiro, Guillermo. 2010.
\newblock Online learning for matrix factorization and sparse coding.
\newblock {\em Journal of Machine Learning Research}, {\bf 11}(Jan), 19--60.

\bibitem[\protect\citename{Mairal {\em et~al.}, }2011]{mairal2011convex}
Mairal, Julien, Jenatton, Rodolphe, Obozinski, Guillaume, \& Bach, Francis.
  2011.
\newblock Convex and network flow optimization for structured sparsity.
\newblock {\em Journal of Machine Learning Research}, {\bf 12}(Sep),
  2681--2720.

\bibitem[\protect\citename{Mangasarian \& Shiau, }1987]{mangasarian1987}
Mangasarian, O.~L., \& Shiau, T-H. 1987.
\newblock Lipschitz continuity of solutions of linear inequalities, programs
  and complementarity problems.
\newblock {\em SIAM Journal on Control and Optimization}, {\bf 25}(3),
  583--595.

\bibitem[\protect\citename{Manzour {\em et~al.}, }2019]{manzour2019integer}
Manzour, Hasan, K{\"u}{\c{c}}{\"u}kyavuz, Simge, \& Shojaie, Ali. 2019.
\newblock Integer Programming for Learning Directed Acyclic Graphs from
  Continuous Data.
\newblock {\em arXiv preprint arXiv:1904.10574}.

\bibitem[\protect\citename{Mazumder \& Radchenko,
  }2017]{mazumder2017thediscrete}
Mazumder, Rahul, \& Radchenko, Peter. 2017.
\newblock TheDiscrete Dantzig Selector: Estimating Sparse Linear Models via
  Mixed Integer Linear Optimization.
\newblock {\em IEEE Transactions on Information Theory}, {\bf 63}(5),
  3053--3075.

\bibitem[\protect\citename{Nesterov, }2013]{nesterov2013gradient}
Nesterov, Yu. 2013.
\newblock Gradient methods for minimizing composite functions.
\newblock {\em Mathematical Programming}, {\bf 140}(1), 125--161.

\bibitem[\protect\citename{Obozinski {\em et~al.}, }2011]{obozinski2011group}
Obozinski, Guillaume, Jacob, Laurent, \& Vert, Jean-Philippe. 2011.
\newblock Group lasso with overlaps: the latent group lasso approach.
\newblock {\em arXiv preprint arXiv:1110.0413}.

\bibitem[\protect\citename{Pang, }1997]{pang1997}
Pang, Jong-Shi. 1997.
\newblock Error bounds in mathematical programming.
\newblock {\em Mathematical Programming}, {\bf 79}(1-3), 299--332.

\bibitem[\protect\citename{Parikh {\em et~al.}, }2014]{parikh2014proximal}
Parikh, Neal, Boyd, Stephen, {\em et~al.} 2014.
\newblock Proximal algorithms.
\newblock {\em Foundations and Trends{\textregistered} in Optimization}, {\bf
  1}(3), 127--239.

\bibitem[\protect\citename{Perrone {\em et~al.}, }2016]{perrone2016poisson}
Perrone, Valerio, Jenkins, Paul~A, Spano, Dario, \& Teh, Yee~Whye. 2016.
\newblock Poisson random fields for dynamic feature models.
\newblock {\em arXiv preprint arXiv:1611.07460}.

\bibitem[\protect\citename{Radchenko \& James, }2010]{radchenko2010variable}
Radchenko, Peter, \& James, Gareth~M. 2010.
\newblock Variable selection using adaptive nonlinear interaction structures in
  high dimensions.
\newblock {\em Journal of the American Statistical Association}, {\bf
  105}(492), 1541--1553.

\bibitem[\protect\citename{Razaviyayn {\em et~al.},
  }2013]{razaviyayn2013unified}
Razaviyayn, Meisam, Hong, Mingyi, \& Luo, Zhi-Quan. 2013.
\newblock A unified convergence analysis of block successive minimization
  methods for nonsmooth optimization.
\newblock {\em SIAM Journal on Optimization}, {\bf 23}(2), 1126--1153.

\bibitem[\protect\citename{Richt{\'a}rik \& Tak{\'a}{\v{c}},
  }2014]{richtarik2014iteration}
Richt{\'a}rik, Peter, \& Tak{\'a}{\v{c}}, Martin. 2014.
\newblock Iteration complexity of randomized block-coordinate descent methods
  for minimizing a composite function.
\newblock {\em Mathematical Programming}, {\bf 144}(1-2), 1--38.

\bibitem[\protect\citename{Schmidt \& Murphy, }2010]{schmidt2010convex}
Schmidt, Mark, \& Murphy, Kevin. 2010.
\newblock Convex structure learning in log-linear models: Beyond pairwise
  potentials.
\newblock {\em Pages  709--716 of:} {\em Proceedings of the Thirteenth
  International Conference on Artificial Intelligence and Statistics}.

\bibitem[\protect\citename{She {\em et~al.}, }2018]{she2018group}
She, Yiyuan, Wang, Zhifeng, \& Jiang, He. 2018.
\newblock Group regularized estimation under structural hierarchy.
\newblock {\em Journal of the American Statistical Association}, {\bf
  113}(521), 445--454.

\bibitem[\protect\citename{Stephen~M., }1973]{robinson1973}
Stephen~M., Robinson. 1973.
\newblock Bounds for error in the solution set of a linear perturbed linear
  program.
\newblock {\em Linear Algebra and Its Applications}, {\bf 6}, 69--81.

\bibitem[\protect\citename{Stephen~M., }1981]{robinson1981}
Stephen~M., Robinson. 1981.
\newblock Some continuity properties polyhedral multifunctions.
\newblock {\em Mathematical Programming Study}, {\bf 14}, 206--214.

\bibitem[\protect\citename{Tibshirani, }1996]{tibshirani1996regression}
Tibshirani, Robert. 1996.
\newblock Regression shrinkage and selection via the lasso.
\newblock {\em Journal of the Royal Statistical Society. Series B
  (Methodological)},  267--288.

\bibitem[\protect\citename{Tseng, }2001]{tseng2001convergence}
Tseng, Paul. 2001.
\newblock Convergence of a block coordinate descent method for
  nondifferentiable minimization.
\newblock {\em Journal of optimization theory and applications}, {\bf 109}(3),
  475--494.

\bibitem[\protect\citename{Tseng, }2010]{tseng2010approximation}
Tseng, Paul. 2010.
\newblock Approximation accuracy, gradient methods, and error bound for
  structured convex optimization.
\newblock {\em Mathematical Programming}, {\bf 125}(2), 263--295.

\bibitem[\protect\citename{Tseng \& Yun, }2009]{tseng2009coordinate}
Tseng, Paul, \& Yun, Sangwoon. 2009.
\newblock A coordinate gradient descent method for nonsmooth separable
  minimization.
\newblock {\em Mathematical Programming}, {\bf 117}(1-2), 387--423.

\bibitem[\protect\citename{Van De~Vijver {\em et~al.}, }2002]{van2002gene}
Van De~Vijver, Marc~J, He, Yudong~D, Van't~Veer, Laura~J, Dai, Hongyue, Hart,
  Augustinus~AM, Voskuil, Dorien~W, Schreiber, George~J, Peterse, Johannes~L,
  Roberts, Chris, Marton, Matthew~J, {\em et~al.} 2002.
\newblock A gene-expression signature as a predictor of survival in breast
  cancer.
\newblock {\em New England Journal of Medicine}, {\bf 347}(25), 1999--2009.

\bibitem[\protect\citename{Villa {\em et~al.}, }2014]{villa2014proximal}
Villa, Silvia, Rosasco, Lorenzo, Mosci, Sofia, \& Verri, Alessandro. 2014.
\newblock Proximal methods for the latent group lasso penalty.
\newblock {\em Computational Optimization and Applications}, {\bf 58}(2),
  381--407.

\bibitem[\protect\citename{Walkup \& Wets, }1969]{walkup1969}
Walkup, David~W., \& Wets, Roger J.-B. 1969.
\newblock A lipschitzian characterization of convex polyhedra.
\newblock {\em Proceedings of the American Mathematical Society}, {\bf 23}(1),
  167--173.

\bibitem[\protect\citename{Wolsey \& Nemhauser, }2014]{wolsey2014integer}
Wolsey, Laurence~A, \& Nemhauser, George~L. 2014.
\newblock {\em Integer and combinatorial optimization}.
\newblock John Wiley \& Sons.

\bibitem[\protect\citename{Yan {\em et~al.}, }2017]{yan2017hierarchical}
Yan, Xiaohan, Bien, Jacob, {\em et~al.} 2017.
\newblock Hierarchical Sparse Modeling: A Choice of Two Group Lasso
  Formulations.
\newblock {\em Statistical Science}, {\bf 32}(4), 531--560.

\bibitem[\protect\citename{Yuan \& Lin, }2006]{yuan2006model}
Yuan, Ming, \& Lin, Yi. 2006.
\newblock Model selection and estimation in regression with grouped variables.
\newblock {\em Journal of the Royal Statistical Society: Series B (Statistical
  Methodology)}, {\bf 68}(1), 49--67.

\bibitem[\protect\citename{Zhang {\em et~al.}, }2013]{zhang2013linear}
Zhang, Haibin, Jiang, Jiaojiao, \& Luo, Zhi-Quan. 2013.
\newblock On the linear convergence of a proximal gradient method for a class
  of nonsmooth convex minimization problems.
\newblock {\em Journal of the Operations Research Society of China}, {\bf
  1}(2), 163--186.

\bibitem[\protect\citename{Zhao \& Yu, }2006]{zhao2006model}
Zhao, Peng, \& Yu, Bin. 2006.
\newblock On model selection consistency of Lasso.
\newblock {\em Journal of Machine learning research}, {\bf 7}(Nov), 2541--2563.

\bibitem[\protect\citename{Zhou \& Man-Cho~So, }2019]{zhou2019}
Zhou, Zirui, \& Man-Cho~So, Anthony. 2019.
\newblock A Unified Approach to Error Bounds for Structured Convex Optimization
  Problems.
\newblock {\em Mathematical Programming}, {\bf 165}(2), 689–--728.

\end{thebibliography}


\appendix
 
\section{Proof of of the dual error bound - Lemma~\ref{lem:dual_err_bound}}\label{app:proof_lem_dual_bound}
The framework of the proof was first proposed in \cite{luo1993convergence} and also applied in \cite{hong2017linear} and requires ``locally upper Lipschitzian" property of \emph{polyhedral} multifunction for the map induced by KKT conditions-- see also \cite{robinson1981,walkup1969,robinson1973,mangasarian1987,hoffman1952}. However, due to the presence of the conic constraints in \eqref{eq:lagrange_min_conic}, the resulting multifunction is \emph{not polyhedral} anymore. Indeed, \cite{walkup1969} showed that having a polyhedral graph is a necessary condition for the upper Lipschitzian property of the multifunction. The following proof uses the specific structure of this problem to establish the dual error bound condition.

For any $\by\in\mR^n$ and $\by^*\in Y^*$, considering the KKT conditions \eqref{eq:kkt}, we have
\begin{align*}
\norm{\by-\by^*}_2^2 &= \norm{M^\top M(\bx^2(\by)-\bx^2(\by^*)) +\rho(\bx^2(\by)-\bx^1(\by))-\rho(\bx^2(\by^*)-\bx^1(\by^*))}_2^2 \\
	&= \norm{\grad_{\bx^2}\phi(M\bx^2(\by))-\grad_{\bx^2}\phi(M\bx^2(\by^*))+\grad_{\bx^2}\psi(E\bx(\by))-\grad_{\bx^2}\psi(E\bx(\by^*))}_2^2 \\
	&\leq  \norm{\grad_{\bx^2}\phi(M\bx^2(\by))-\grad_{\bx^2}\phi(M\bx^2(\by^*))}_2^2+\norm{\grad_{\bx}\psi(E\bx(\by))-\grad_{\bx}\psi(E\bx(\by^*))}_2^2, \\
\end{align*}
where the first equality follows from \eqref{eq:kkt_2}, the second equality follows from the definition of $\phi(\cdot)$ and $\psi(\cdot)$ (defined below \eqref{eq:ell_func}), the third inequality follows from the triangle inequality. Hence, using \eqref{eq:prop_phi} and \eqref{eq:prop_psi}, we have
\begin{equation}
\norm{\by-\by^*}_2^2 \leq L_{\phi}^2\norm{M\bx^2(\by)-M\bx^2(\by^*)}_2^2+L_{\psi}^2\norm{E\bx(\by)-E\bx(\by^*)}_2^2. \label{eq:dual_336}
\end{equation}
Next, consider
{\small
\begin{align*}
&\norm{M\bx^2(\by)-M\bx^2(\by^*)}_2^2+\rho\norm{E\bx(\by)-E\bx(\by^*)}_2^2  \\
	&= \fprod{M^\top M\bx^2(\by)-M^\top M\bx^2(\by^*),\bx^2(\by)-\bx^2(\by^*)}+\rho \fprod{E^\top E\bx(\by)-M^\top M\bx(\by^*),\bx(\by)-\bx(\by^*)} \\
	&= \fprod{\grad_{\bx}\phi(M\bx^2(\by))-\grad_{\bx}\phi(M\bx^2(\by^*)),\bx(\by)-\bx(\by^*)}+ \fprod{\grad_{\bx}\psi(E\bx(\by))-\grad_{\bx}\psi(E\bx(\by^*)),\bx(\by)-\bx(\by^*)} \\
	& =\fprod{\grad_{\bx}\ell(\bx(\by))-\grad_{\bx}\ell(\bx(\by^*)),\bx(\by)-\bx(\by^*)} \\
	& =\fprod{\grad_{\bx^1}\ell(\bx(\by))-\grad_{\bx^1}\ell(\bx(\by^*)),\bx^1(\by)-\bx^1(\by^*)} + \fprod{\grad_{\bx^2}\ell(\bx(\by))-\grad_{\bx^2}\ell(\bx(\by^*)),\bx^2(\by)-\bx^2(\by^*)} \\
	& =\sum_{g\in\cG}\fprod{\grad_{\bx^1_{j(g)}}\ell(\bx(\by))-\grad_{\bx^1_{j(g)}}\ell(\bx(\by^*)),\bx^1_{j(g)}(\by)-\bx^1_{j(g)}(\by^*)} \\
	&\qquad + \sum_{g\in\cG}\fprod{\grad_{\bx^2_{j(g)}}\ell(\bx(\by))-\grad_{\bx^2_{j(g)}}\ell(\bx(\by^*)),\bx^2_{j(g)}(\by)-\bx^2_{j(g)}(\by^*)} \\
	&= \sum_{g\in\cG}\fprod{w_g\m_g(\by)-\by_{j(g)}-w_g\m_g(\by^*)+\by^*_{j(g)},\bx^1_{j(g)}(\by)-\bx^1_{j(g)}(\by^*)}\\
	&\qquad +\sum_{g\in\cG}\fprod{\by_{j(g)}-\by_{j(g)}^*,\bx^2_{j(g)}(\by)-\bx^2_{j(g)}(\by^*)}
\end{align*}
}
where the second and third equalities follow from the definitions of $\phi$, $\psi$, and $\ell$, in the forth and fifth equalities the gradient is expanded over each $\bx^1$ and $\bx^2$, and the last equality follows from \eqref{eq:kkt_1},\eqref{eq:kkt_2}. Rearranging the terms in the last line above and using $\bx_{j(g)}^1(\by^*)=\bx_{j(g)}^2(\by^*)\ \forall g\in\cG$, we get
{\small
\begin{equation}
\begin{split}
&\norm{M\bx^2(\by)-M\bx^2(\by^*)}_2^2+\rho\norm{E\bx(\by)-E\bx(\by^*)}_2^2 \\
&\quad = \sum_{g\in\cG}\fprod{w_g\m_g(\by)-w_g\m_g(\by^*),\bx^1_{j(g)}(\by)-\bx^1_{j(g)}(\by^*)}+\sum_{g\in\cG}\fprod{\by_{j(g)}-\by^*_{j(g)},\bx^2_{j(g)}(\by)-\bx^1_{j(g)}(\by)}, \label{eq:45}
\end{split}
\end{equation}
}

For all $g\in\cG$, we have
\begin{align*}
&\fprod{w_g\m_g(\by)-w_g\m_g(\by^*),\bx^1_{j(g)}(\by)-\bx^1_{j(g)}(\by^*)} \\
&\quad =  w_g\m_g(\by)^\top\bx^1_{j(g)}(\by)-w_g\m_g(\by^*)^\top\bx^1_{j(g)}(\by) - w_g\m_g(\by)^\top\bx^1_{j(g)}(\by^*)+w_g\m_g(\by^*)^\top\bx^1_{j(g)}(\by^*) \\
&\quad= -w_g(s_g/w_g)\lambda-w_g\m_g(\by^*)^\top\bx^1_{j(g)}(\by)-w_g\m_g(\by)^\top\bx^1_{j(g)}(\by^*)-w_g(s_g/w_g)\lambda \\
&\quad\leq w_g\norm{\m_g(\by^*)}_2\norm{\bx^1_{j(g)}(\by)}_2+w_g\norm{\m_g(\by)}_2\norm{\bx^1_{j(g)}(\by^*)}_2-2s_g\lambda \\
&\quad\leq 0, 
\end{align*}
where the second equality follows from \eqref{eq:kkt_3} and \eqref{eq:kkt_6}, the third inequality uses Cauchy-Schwarz inequality, and  the last inequality follows from \eqref{eq:kkt_3}, \eqref{eq:kkt_4}, and \eqref{eq:kkt_5}. Hence, we have
{\small
\begin{align}
\norm{M\bx^2(\by)-M\bx^2(\by^*)}_2^2+\rho\norm{E\bx(\by)-E\bx(\by^*)}_2^2 &\leq \sum_{g\in\cG} \fprod{\by_{j(g)}-\by^*_{j(g)},\bx^2_{j(g)}(\by)-\bx^1_{j(g)}(\by)} \label{eq:551} \\
	&\leq \norm{\by-\by^*}\norm{\grad g_{\rho}(\by)}, \label{eq:52}
\end{align}
}
where \eqref{eq:551} uses nonpositivity of the first term in \eqref{eq:45} (shown above), and \eqref{eq:52} follows from Cauchy-Schwarz inequality and the fact that  $\grad g_{\rho}(\by)=\bx^1(\by)-\bx^2(\by)$ -- see e.g. \cite{hong2017linear} - Lemma~2.1. Finally, using \eqref{eq:52} and \eqref{eq:dual_336}, we get
\begin{align}
\norm{\by-\by^*}_2^2  &\leq \max\{L_{\phi}^2,L_{\psi}^2/\rho\}\Big(\norm{M\bx^2(\by)-M\bx^2(\by^*)}_2^2+\rho\norm{E\bx(\by)-E\bx(\by^*)}_2^2\Big) \\
		&\leq \max\{L_{\phi}^2,L_{\psi}^2/\rho\}\norm{\by-\by^*}_2\norm{\grad g_{\rho}(\by)}_2.
\end{align}
Hence, we have
\begin{equation}
\text{dist}(\by,Y^*) \leq \norm{\by-\by^*} \leq \max\{L_{\phi}^2,L_{\psi}^2/\rho\}\norm{\grad g_{\rho}(\by)}_2.
\end{equation}

\section{Proof of Lemma~\ref{lem:boundedness}}\label{app:proof_lem_boundedness}
Before proceeding with the proof of the boundedness of the iterates, we show the existence of a finite saddle point by the following argument. Consider 
\begin{equation}
\label{AAA}
\min_{\bx^1,\bx^2\in\mR^n} \left\{ \tilde{F}(\bx^1,\bx^2)=\lambda\sum_{g\in\cG}w_g\norm{\bx^1_{j(g)}}_2+\frac{1}{2}\norm{M\bx^2-\bb}_2^2+\frac{\rho}{2}\norm{\bx^1-\bx^2}_2^2, \ \ \text{s.t.} \ \ \bx^1=\bx^2 \right\}.
\end{equation}
The above problem is equivalent to \eqref{eq:main_2} whose objective function is coercive and continuous. By Weierstrass's Theorem (see e.g.~\cite{bertsekas1999nonlinear}), we have certain finite optimal solution to \eqref{AAA} $(\bx^{1,*},\bx^{2,*})$, i.e. $\tilde{F}^{*}=\inf _{\bx^{1}=\bx^{2}\in R^{n}} \tilde{F}(\bx^{1}, \bx^{2})=\tilde{F}(\bx^{1,*},\bx^{2,*})$. Especially, $\bx^{1,*}=\bx^{2,*}$.
Consider $L_\rho(\bx^1,\bx^2;\by)=\tilde{F}(\bx^1,\bx^2)+\fprod{\by,\bx^1-\bx^2}$ and the corresponding dual function $g_\rho(\by)$. First, we have $L_\rho(\bx^{1,*},\bx^{2,*};\by)=\tilde{F}(\bx^{1,*},\bx^{2,*})$, for $\forall \by$. By strong duality (see e.g. Prop. 5.2.1 in \cite{bertsekas1999nonlinear}), we know there is no duality gap. Furthermore, there exists at least one Lagrange multiplier $\by^*$, i.e. $\tilde{F}^*=\inf_{\bx^1,\bx^2\in\mR} L_\rho(\bx^1,\bx^2;\by^*)=g_\rho(\by^*)$. We conclude $(\bx^{1,*},\bx^{2,*};\by)$ is a finite saddle point.

The idea for the proof of boundedness of the iterates is similar to Theorem 5.1 in \cite{glowinski}; however, we do not have the strong convexity assumption. Given a finite saddle point $((\bx^{1,*},\bx^{2,*}); \by^*)$, define $\tilde{\bx}^{1,k}\triangleq \bx^{1,k}-\bx^{1,*}$, $\tilde{\bx}^{2,k}\triangleq \bx^{2,k}-\bx^{2,*}$, $\tilde{\by}^{k}\triangleq \by^{k}-\by^{*}$ where $\bx^{1,*}=\bx^{2,*}$. Establishing the boundedness of the sequence is equivalent to showing that the sequence $\{\norm{\tilde{\by}^{k}}^2_2+\alpha\rho\norm{\tilde{\bx}^{2,k}}^2_2+\alpha(\rho-\alpha)\norm{\tilde{\bx}^{1,k}-\tilde{\bx}^{2,k}}^2_2\}_{k=1}^{\infty}$ is non-increasing. From the convexity of the augmented Lagrangian function \eqref{eq:augLagrange} in $\bx^1$, we have
\begin{equation}\label{1}
    \fprod{\by^*+\rho(\bx^{1,*}-\bx^{2,*}), \bx^1-\bx^{1,*}}+\lambda\sum_{g\in\cG}w_g\norm{\bx^1_{j(g)}}_2-\lambda\sum_{g\in\cG}w_g\norm{\bx^{1,*}_{j(g)}}_2 \geq 0, \forall \bx^1. 
\end{equation}
Furthermore, from the convexity of the augmented Lagrangian \eqref{eq:augLagrange} in $\bx^2$, we have
\begin{equation}\label{2}
    \fprod{M^\top(M\bx^{2,*}-b)-\by^*+\rho(\bx^{2,*}-\bx^{1,*}), \bx^2-\bx^{2,*}}\geq 0, \forall \bx^2.
\end{equation}
From the fact that $L_\rho(\bx^{1,*},\bx^{2,*};\by^*)\leqslant L_\rho(\bx^1,\bx^2;\by^*), \forall \bx^1, \bx^2$, we have
\begin{equation}\label{3}
    \by^*=\by^*+\alpha(\bx^{1,*}-\bx^{2,*}).
\end{equation}
Similar to the arguments for \eqref{1}-\eqref{3}, from \eqref{eq:admm_0_1}-\eqref{eq:admm_0_3}, we have
{\small
\begin{align}
    &\fprod{\rho(\bx^{1,k+1}_{j(g)}-\bx^{2,k}_{j(g)})+\by^k_{j(g)}, \bx^1_{j(g)}-\bx^{1,k+1}_{j(g)}}+\lambda w_g\norm{\bx^1_{j(g)}}_2-\lambda w_g\norm{\bx^{1,k+1}_{j(g)}}_2 \geq 0, \ \ \forall g\in\cG, \forall \bx^1_{j(g)}, \label{4}\\
    &\fprod{M^\top(M\bx^{2,k+1}-b)+\rho(\bx^{2,k+1}-\bx^{1,k+1})-\by^k, \bx^2-\bx^{2,k+1}} \geq 0, \ \ \forall \bx^2, \label{5}\\
    &\by^{k+1}_{j(g)}=\by^k_{j(g)}+\alpha(\bx^{1,k+1}_{j(g)}-\bx^{2,k+1}_{j(g)}),\ \ \forall g\in\cG. \label{6}
\end{align}
}
Since $j(g)\cap j(\bar{g})=\emptyset$ for all $g,\bar{g}\in\cG$ such that $g\neq\bar{g}$, from \eqref{4} and \eqref{6}, we have: 
\begin{align}
     &\fprod{\rho(\bx^{1,k+1}-\bx^{2,k})+\by^k, \bx^1-\bx^{1,k+1}}+\lambda\sum_{g\in\cG}w_g\norm{\bx^1_{j(g)}}_2-\lambda\sum_{g\in\cG}w_g\norm{\bx^{1,k+1}_{j(g)}}_2 \geq 0, \ \ \forall \bx^1, \label{7}\\
    &\by^{k+1}=\by^k+\alpha(\bx^{1,k+1}-\bx^{2,k+1}).\ \ \label{8}
\end{align}
Setting $\bx^1=\bx^{1,k+1}$ in \eqref{1}, and $\bx^1=\bx^{1,*}$ in \eqref{7} and adding them, we get
\begin{equation}\label{9}
    \fprod{-\tilde{\by}^k+\rho(\tilde{\bx}^{2,k}-\tilde{\bx}^{1,k+1}),\tilde{\bx}^{1,k+1}}\geq 0
\end{equation}
Similarly, setting $\bx^2=\bx^{2,k+1}$ in \eqref{2}, and $\bx^2=\bx^{2,*}$ in \eqref{5} and adding them, we get
\begin{equation}\label{10}
\fprod{M^\top M\tilde{\bx}^{2,k+1}-\tilde{\by}^k-\rho(\tilde{\bx}^{1,k+1}-\tilde{\bx}^{2,k+1}), -\tilde{\bx}^{2,k+1} }\geq 0
\end{equation} 
Adding the left-hand-sides of \eqref{9} to \eqref{10} and rearranging the terms, we have, 
{\small
\begin{align*}
&\fprod{-\tilde{\by}^k,\tilde{\bx}^{1,k+1}-\tilde{\bx}^{2,k+1}}+\rho\fprod{\tilde{\bx}^{2,k}-\tilde{\bx}^{1,k+1},\tilde{\bx}^{1,k+1}}+\rho\fprod{\tilde{\bx}^{1,k+1}-\tilde{\bx}^{2,k+1},\tilde{\bx}^{2,k+1}}-\norm{M\tilde{\bx}^{2,k+1}}^2_2 \\
&=\fprod{-\tilde{\by}^k,\tilde{\bx}^{1,k+1}-\tilde{\bx}^{2,k+1}}
+\rho\fprod{\tilde{\bx}^{2,k}-\tilde{\bx}^{1,k+1}+\tilde{\bx}^{2,k+1}-\tilde{\bx}^{2,k+1},\tilde{\bx}^{1,k+1}} \\
&+\rho\fprod{\tilde{\bx}^{1,k+1}-\tilde{\bx}^{2,k+1}),\tilde{\bx}^{2,k+1}}-\norm{M\tilde{\bx}^{2,k+1}}^2_2
\end{align*}
}
Hence, we have
{\small
\begin{equation}\label{A}
\rho\fprod{\tilde{\bx}^{2,k}-\tilde{\bx}^{2,k+1},\tilde{\bx}^{1,k+1}}-\norm{M\tilde{\bx}^{2,k+1}}^2_2-\rho\norm{\tilde{\bx}^{2,k+1}-\tilde{\bx}^{1,k+1}}_2^2 \geq \fprod{\tilde{\bx}^{1,k+1}-\tilde{\bx}^{2,k+1},\tilde{\by}^k}.
\end{equation}
}
From \eqref{8}, we have the following two inequalities: 
\begin{align*}
    \tilde{\by}^{k+1}-\tilde{\by}^k&=\alpha(\tilde{\bx}^{1,k+1}-\tilde{\bx}^{2,k+1})\\
    \tilde{\by}^{k+1}+\tilde{\by}^k&=\alpha(\tilde{\bx}^{1,k+1}-\tilde{\bx}^{2,k+1})+2\tilde{\by}^k
\end{align*}
Taking the inner product of the left terms together and the right terms together, we obtain
\begin{align}\label{B}
&\norm{\tilde{\by}^{k+1}}^2_2-\norm{\tilde{\by}^{k}}^2_2 = \alpha^2\norm{\tilde{\bx}^{1,k+1}-\tilde{\bx}^{2,k+1}}+2\alpha\fprod{\tilde{\bx}^{1,k+1}-\tilde{\bx}^{2,k+1},\tilde{\by}^k}\\
&\leq \alpha(\alpha-2\rho)\norm{\tilde{\bx}^{1,k+1}-\tilde{\bx}^{2,k+1}}^2_2-2\alpha\norm{M\tilde{\bx}^{2,k+1}}_2^2+2\alpha\rho\fprod{\tilde{\bx}^{2,k}-\tilde{\bx}^{2,k+1},\tilde{\bx}^{1,k+1}}
\end{align}
where the inequality uses \eqref{A}. Next, we will upper bound $\fprod{\tilde{\bx}^{2,k}-\tilde{\bx}^{2,k+1},\tilde{\bx}^{1,k+1}}$. Setting $\bx^2=\bx^{2,k}$ in \eqref{5}, we have
\begin{equation}\label{C1}
\fprod{M^\top(M\bx^{2,k+1}-b)+\rho(\bx^{2,k+1}-\bx^{1,k+1})-\by^k, \bx^{2,k}-\bx^{2,k+1}} \geq 0.
\end{equation}
Setting $k+1$ in \eqref{5} to $k$, and $\bx^2=\bx^{2,k+1}$, we have
\begin{equation}\label{C2}
\fprod{M^\top(M\bx^{2,k}-b)+\rho(\bx^{2,k}-\bx^{1,k})-\by^{k-1}, \bx^{2,k+1}-\bx^{2,k}} \geq 0.
\end{equation} 
Adding \eqref{C1} and \eqref{C2}, we have
\begin{align}\label{D}
\begin{split}
&\fprod{\by^k-\by^{k-1},\bx^{2,k+1}-\bx^{2,k}}-\rho\norm{\bx^{2,k+1}-\bx^{2,k}}^2_2+\rho\fprod{\bx^{1,k+1}-\bx^{1,k},\bx^{2,k+1}-\bx^{2,k}}\\
&\qquad \geq\norm{M(\bx^{2,k+1}-\bx^{2,k})}_2^2\geq 0.
\end{split}
\end{align}
From \eqref{8}, we have $\by^k-\by^{k-1}=\alpha(\bx^{1,k}-\bx^{2,k})$. Using it in \eqref{D} and rearranging terms, we obtain 
\begin{equation*}\label{E}
    \rho\fprod{\bx^{1,k+1}-\bx^{1,k},\bx^{2,k+1}-\bx^{2,k}}\geq \rho\norm{\bx^{2,k+1}-\bx^{2,k}}^2_2-\alpha\fprod{\bx^{1,k}-\bx^{2,k}, \bx^{2,k+1}-\bx^{2,k}}.
\end{equation*}
Adding and subtracting $\bx^{1,*}$ and $\bx^{2,*}$ into each argument in \eqref{E} as needed, we have
\begin{equation}\label{F}
    \rho\fprod{\tilde{\bx}^{1,k+1}-\tilde{\bx}^{1,k},\tilde{\bx}^{2,k+1}-\tilde{\bx}^{2,k}}\geq \rho\norm{\tilde{\bx}^{2,k+1}-\tilde{\bx}^{2,k}}^2_2-\alpha\fprod{\tilde{\bx}^{1,k}-\tilde{\bx}^{2,k}, \tilde{\bx}^{2,k+1}-\tilde{\bx}^{2,k}}.
\end{equation}
The term $\fprod{\tilde{\bx}^{2,k}-\tilde{\bx}^{2,k+1},\tilde{\bx}^{1,k+1}}$ can be transformed as following:
{\small
\begin{align}\label{G}
\begin{split}
&\fprod{\tilde{\bx}^{2,k}-\tilde{\bx}^{2,k+1},\tilde{\bx}^{1,k+1}}\\
&=\fprod{\tilde{\bx}^{2,k}-\tilde{\bx}^{2,k+1},\tilde{\bx}^{1,k+1}-\tilde{\bx}^{1,k}+\tilde{\bx}^{1,k}-\tilde{\bx}^{2,k}+\tilde{\bx}^{2,k}}\\
&=\fprod{\tilde{\bx}^{2,k}-\tilde{\bx}^{2,k+1},\tilde{\bx}^{1,k+1}-\tilde{\bx}^{1,k}}+\fprod{\tilde{\bx}^{2,k}-\tilde{\bx}^{2,k+1},\tilde{\bx}^{1,k}-\tilde{\bx}^{2,k}}+\fprod{\tilde{\bx}^{2,k}-\tilde{\bx}^{2,k+1},\tilde{\bx}^{2,k}}\\
&=\fprod{\tilde{\bx}^{2,k}-\tilde{\bx}^{2,k+1},\tilde{\bx}^{1,k+1}-\tilde{\bx}^{1,k}}+\fprod{\tilde{\bx}^{2,k}-\tilde{\bx}^{2,k+1},\tilde{\bx}^{1,k}-\tilde{\bx}^{2,k}}+\frac{1}{2}(\norm{\tilde{\bx}^{2,k}}^2_2-\norm{\tilde{\bx}^{2,k+1}}^2_2+\norm{\tilde{\bx}^{2,k}-\tilde{\bx}^{2,k+1}}^2_2)\\
&\leq \frac{1}{2}(\norm{\tilde{\bx}^{2,k}}^2_2-\norm{\tilde{\bx}^{2,k+1}}^2_2-\norm{\tilde{\bx}^{2,k}-\tilde{\bx}^{2,k+1}}^2_2)+(1-\frac{\alpha}{\rho})\fprod{\tilde{\bx}^{2,k}-\tilde{\bx}^{2,k+1},\tilde{\bx}^{1,k}-\tilde{\bx}^{2,k}},
\end{split}
\end{align}
}
where the last inequality follows from \eqref{F}. Combining \eqref{B} and \eqref{G} and rearranging the terms, we obtain
{\small
\begin{align}\label{H}
\begin{split}
&\norm{\tilde{\by}^{k+1}}_2^2+\alpha\rho\norm{\tilde{\bx}^{2,k+1}}_2^2+\alpha(\rho-\alpha)\norm{\tilde{\bx}^{1,k+1}-\tilde{\bx}^{2,k+1}}^2_2-(\norm{\tilde{\by}^{k}}_2^2+\alpha\rho\norm{\tilde{\bx}^{2,k}}_2^2)\\
&\leq -\alpha\rho\norm{\tilde{\bx}^{1,k+1}-\tilde{\bx}^{2,k+1}}^2_2-2\alpha\norm{M\tilde{\bx}^{2,k+1}}^2_2-\alpha\rho\norm{\tilde{\bx}^{2,k}-\tilde{\bx}^{2,k+1}}^2_2+2\alpha(\rho-\alpha)\fprod{\tilde{\bx}^{2,k}-\tilde{\bx}^{2,k+1},\tilde{\bx}^{1,k}-\tilde{\bx}^{2,k}}
\end{split}
\end{align}
}
By upper bounding the last term in \eqref{H} by the identity $2\fprod{\ba,\bb}\leq\norm{\ba}_2^2+\norm{\bb}_2^2$, we get
{\small
\begin{align}
\begin{split}
&\norm{\tilde{\by}^{k+1}}_2^2+\alpha\rho\norm{\tilde{\bx}^{2,k+1}}_2^2+\alpha(\rho-\alpha)\norm{\tilde{\bx}^{1,k+1}-\tilde{\bx}^{2,k+1}}^2_2-(\norm{\tilde{\by}^{k}}_2^2+\alpha\rho\norm{\tilde{\bx}^{2,k}}_2^2+\alpha(\rho-\alpha)\norm{\tilde{\bx}^{1,k}-\tilde{\bx}^{2,k}}^2_2)\\
&\leq -\alpha\rho\norm{\tilde{\bx}^{1,k+1}-\tilde{\bx}^{2,k+1}}^2_2-2\alpha\norm{M\tilde{\bx}^{2,k+1}}^2_2-\alpha^2\norm{\tilde{\bx}^{2,k}-\tilde{\bx}^{2,k+1}}^2_2\leq 0. 
\end{split}
\end{align}
}
We have shown that the sequence $\{\norm{\tilde{\by}^{k}}^2_2+\alpha\rho\norm{\tilde{\bx}^{2,k}}^2_2+\alpha(\rho-\alpha)\norm{\tilde{\bx}^{1,k}-\tilde{\bx}^{2,k}}^2_2\}_{k=1}^{\infty}$ is non-increasing. Once the initial point and saddle point are fixed, which are not related to $\alpha$, then $\{\norm{\tilde{\by}^{k}}^2_2+\alpha\rho\norm{\tilde{\bx}^{2,k}}^2_2+\alpha(\rho-\alpha)\norm{\tilde{\bx}^{1,k}-\tilde{\bx}^{2,k}}^2_2\}_{k=1}^{\infty}$ is bounded by $\norm{\tilde{\by}^{0}}^2_2+\rho^2\norm{\tilde{\bx}^{2,0}}^2_2+\frac{\rho^2}{4}\norm{\tilde{\bx}^{1,0}-\tilde{\bx}^{2,0}}^2_2$. We concluded that the sequence $\{\bx^{1,k}\}$, $\{\bx^{2,k}\}$ and $\{\by^k\}$ generated by \eqref{eq:admm_0_3} is uniformly bounded for any $0<\alpha<\rho$.

\section{Proof of Lemma~\ref{lem:primal_err_bound}}\label{app:proof_lem_primal_bound}
The proof extends the analysis of \cite{tseng2010approximation} and \cite{zhang2013linear}. Since both works discuss primal methods, there are mainly two new ingredients in our proof: 1) dealing with the dual variable $\by$, and 2) splitting $\bx$ into $\bx^1$ and $\bx^2$, where neither step is trivial.

Given $\by$, note that $\bX(\by)$ can be written as $(\bX^1(\by),\bX^2(\by))$. 
For a fixed $\by$, and for \emph{any} sequence $\{(\bx^{1,k},\bx^{2,k};\by): \bx^{2,k} \not\in \bX^2(\by)\}_{k\geq 0}$, we define 
\begin{align}
    \br^{1,k}&\triangleq \tilde{\nabla}_{\bx^1} L_{\rho}(\bx^{1,k},\bx^{2,k};\by), \label{eq:r11}\\
    \br^{2,k}&\triangleq \tilde{\nabla}_{\bx^2} L_{\rho}(\bx^{1,k},\bx^{2,k};\by)=M^T(M\bx^{2,k}-\bb)-\by+\rho(\bx^{2,k}-\bx^{1,k}),\label{r2}\\
    \delta^k&\triangleq \norm{\bx^{2,k}-\bar{\bx}^{2,k}}_2, \ \text{ where } \bar{\bx}^{2,k}\triangleq \argmin_{\bx^2\in \bX^2(\by)}\norm{\bx^{2,k}-\bx^2}_2, \label{delta}\\
    \bar{\bx}^{1,k}&\triangleq \argmin_{\bx^1}L_{\rho}(\bx^{1},\bar{\bx}^{2,k};\by), \label{x_1}\\
    \bu^k &\triangleq \frac{\bx^{2,k}-\bar{\bx}^{2,k}}{\delta^k}.\label{uk}
\end{align}
Note that
\small{ 
\begin{align}
    \tilde{\nabla}_{\bx^1} L_{\rho}(\bx^{1},\bx^{2};\by)&=\bx^1-\operatorname{prox}_{\lambda\sum_{g\in\cG}w_g\norm{\tilde{\bd}_{j(g)}}_2}(\bx^1-\by-\rho (\bx^1-\bx^2))\\
    &=\bx^1-\argmin_{\tilde{\bd}}\ \lambda\sum_{g\in\cG} w_g\norm{\tilde{\bd}_{j(g)}}_2 +\frac{1}{2}\norm{\tilde{\bd}-(\bx^1-\by+\rho(\bx^1-\bx^2))}_2^2 \\
    &=\argmin_{\bd} \lambda\sum_{g\in\cG}w_g\norm{\bd_{j(g)}-\bx^1_{j(g)}}_2+ \frac{1}{2}\norm{\bd-\by-\rho(\bx^1-\bx^2)}^2_2, \label{eq:r1d}
\end{align}
where the second equality follows from the definition of the proximal operator and the third equality uses the transformation $\bd\triangleq\bx^1-\tilde{\bd}$. Furthermore, for any group $g\in\cG$, we have
\begin{align}\label{r1}
    &\left(\tilde{\nabla}_{\bx^1} L_{\rho}(\bx^{1},\bx^{2};\by)\right)_{j(g)}=\argmin_{\bd_{j(g)}} \lambda w_g\norm{\bd_{j(g)}-\bx^1_{j(g)}}_2+ \frac{1}{2}\norm{\bd_{j(g)}-\by_{j(g)}-\rho(\bx^1_{j(g)}-\bx^2_{j(g)})}^2_2\\
    &=
    \begin{cases}\label{cc}
        \bx^1_{j(g)}, \ \qquad & \text{if} \ \  \norm{\bx^1_{j(g)}-\by_{j(g)}-\rho(\bx^1_{j(g)}-\bx^2_{j(g)})}_2 \leq \lambda w_g,\\
        \gamma_g\bx^1_{j(g)}+(1-\gamma_g)(\by_{j(g)}+\rho(\bx^1_{j(g)}-\bx^2_{j(g)})), \ &\text{otherwise.}
    \end{cases}
\end{align}
}
where $\gamma_g=\lambda w_g/\norm{\bx^1_j(g)-\by_{j(g)}-\rho(\bx^1_{j(g)}-\bx^2_{j(g)})}_2$. Note that the two cases from the soft-thresholding operator in \eqref{cc} yield $\bx^1_{j(g)}$ at the boundary $\norm{\bx^1_{j(g)}-\by_{j(g)}-\rho(\bx^1_{j(g)}-\bx^2_{j(g)})}_2=\lambda w_g$, i.e., $\tilde{\nabla}_{x^1} L_{\rho}(\bx^{1},\bx^{2};\by)_{j(g)}$ is continuous in $(\bx^1,\bx^2,\by)$.  

To prove this lemma, we will first prove that it suffices to show that there exists $0<\tau'<+\infty$ and $\delta>0$ such that 
\begin{equation}\label{primal_error_small}
\text{dist}(\bx^2,\bX^2(\by))\leq \tau'\norm{\tilde{\nabla}_{\bx} L_{\rho}(\bx^1,\bx^2;\by)}_2,
\end{equation}
for all $(\bx^1,\bx^2;\by)$ such that $\norm{\tilde{\nabla}_{x} L_{\rho}(\bx^1,\bx^2;\by)}_2\leq \delta$. Second, we will show \eqref{primal_error_small}.

Assume \eqref{primal_error_small} holds. Given $(\bx^1,\bx^2;\by)$, pick $(\bx^{1,*},\bx^{2,*})\in \bX(\by)$, such that $\text{dist}(\bx^2,\bx^{2,*})=\text{dist}(\bx^2,\bX^2(\by))$, and $\bx^{1,*}$ such that it satisfies the optimality condition \eqref{12}. Recall that  
\begin{equation}
\tilde{\nabla}_{\bx^2} L_{\rho}(\bx^1,\bx^2;\by)=M^T(M\bx^2-\bb)-\by+\rho(\bx^2-\bx^1). \label{11}
\end{equation}
Hence, from the optimality condition, we have
\begin{equation}
\tilde{\nabla}_{\bx^2} L_{\rho}(\bx^{1,*},\bx^{2,*};\by)=M^T(M\bx^{2,*}-\bb)-\by+\rho(\bx^{2,*}-\bx^{1,*})=0\label{12}
\end{equation}
Subtracting \eqref{12} from \eqref{11} and rearranging the terms, we obtain
\begin{equation}\label{limit}
    \bx^1-\bx^{1,*}=(\frac{1}{\rho} M^T M+\bI)(\bx^2-\bx^{2,*})-\frac{1}{\rho}\tilde{\nabla}_{x^2} L_{\rho}(\bx^1,\bx^2;\by).
\end{equation}
Thus, 
\begin{align}
 \text{dist}(\bx,\bX(\by))^2&\leq \norm{\bx^1-\bx^{1,*}}^2_2+\norm{\bx^2-\bx^{2,*}}^2_2\\
 &\leq\norm{(\frac{1}{\rho} M^T M+\bI)(\bx^2-\bx^{2,*})}^2_2+\norm{\frac{1}{\rho}\tilde{\nabla}_{x^2} L_{\rho}(\bx^1,\bx^2;\by)}_2^2+\norm{\bx^2-\bx^{2,*}}^2_2. \label{eq:77}
\end{align}
Upper bounding  $\norm{\bx^2-\bx^{2,*}}^2_2$ in \eqref{eq:77} with \eqref{primal_error_small}, we have \eqref{primal_error}.

Next, we will show \eqref{primal_error_small} by contradiction. Suppose \eqref{primal_error_small} does not hold, then there exists a sequence $\{(\bx^{1,k},\bx^{2,k};\by): \bx^{2,k} \not\in \bX^2(\by)\}_{k\geq 0}$ satisfying 
\begin{equation}\label{eq:78}
\norm{\tilde{\nabla}_\bx L_{\rho}(\bx^{1,k},\bx^{2,k},\by)}_2/\delta^k\rightarrow 0, \quad \text{ and } \quad \norm{\tilde{\nabla}_\bx L_{\rho}(\bx^{1,k},\bx^{2,k},\by)}_2\rightarrow 0.
\end{equation}
Note that
\begin{equation}\label{eq:79}
    \frac{\norm{\br^{1,k}}+\norm{\br^{2,k}}}{\sqrt{2}}\leq \norm{\tilde{\nabla}_{x} L_{\rho}(\bx^{1,k},\bx^{2,k};\by)}\leq \norm{\br^{1,k}}+\norm{\br^{2,k}}
\end{equation}
where $\br^{1,k}$ and $\br^{2,k}$ are defined in \eqref{eq:r11} and \eqref{r2}, respectively. Hence, using the left inequality in \eqref{eq:79}, \eqref{eq:78} implies
\begin{equation}\label{condtion}
    \{\br^{1,k}\} \rightarrow \boldsymbol{0},\ \ \{\br^{2,k}\} \rightarrow \boldsymbol{0}, \ \ \{\frac{\norm{\br^{1,k}}+\norm{\br^{2,k}}}{\delta^k}\} \rightarrow 0.
\end{equation}
We will show that \eqref{condtion} does not hold. Since $(\bx^{1,k},\bx^{2,k})$ is in a compact set, by passing to a subsequence if necessary, we can assume that $\{(\bx^{1,k},\bx^{2,k})\rightarrow (\bar{\bx}^1,\bar{\bx}^2)\}$.  Since $\{\br^{1,k}\} \rightarrow \boldsymbol{0}$, and $\{\br^{2,k}\} \rightarrow \boldsymbol{0}$, then by the right inequality in \eqref{eq:79}, $\tilde{\nabla}_{x} L_{\rho}(\bx^{1,k},\bx^{2,k};\by)\rightarrow\bo$. Furthermore, since $\tilde{\nabla}_{x} L_{\rho}(\bx^1,\bx^2;\by)$ is continuous, this implies $\tilde{\nabla}_{x} L_{\rho}(\bar{x}^1,\bar{x}^2;\by)=\boldsymbol{0}$. It further implies that $(\bar{\bx}^1,\bar{\bx}^2)\in\bX(\by)$. Hence $\delta^k\leq \norm{\bx^{2,k}-\bar{\bx}^2}\rightarrow 0$, as $k\rightarrow \infty$, so that $\{\bar{\bx}^{2,k}\}\rightarrow \bar{\bx}^2$. And based on \eqref{limit}, we have 
\begin{equation}\label{limitt}
    \{\bar{\bx}^{1,k}\}\rightarrow \bar{\bx}^1.
\end{equation}\\
Next, we claim there exists $\kappa>0$ such that, 
\begin{equation}\label{key}
    \norm{\bx^{2,k}-\bar{\bx}^{2,k}}\leq \kappa \norm{M\bx^{2,k}-M\bar{\bx}^{2,k}}, \ \ \forall k
\end{equation}
Again, we argue \eqref{key} by contraction. Suppose \eqref{key} does not hold, then by passing to a subsequence if necessary, we can assume
\begin{equation}\label{key_contradict}
    \{\frac{\norm{M\bx^{2,k}-M\bar{\bx}^{2,k}}}{\norm{\bx^{2,k}-\bar{\bx}^{2,k}}}\}\rightarrow 0.
\end{equation}
This implies that $\{M\bu^k\}\rightarrow 0$, where $\bu^k$ is defined in \eqref{uk}. Note that $\norm{\bu^k}=1$, we can assume $\bu^k \rightarrow \bar{\bu} \neq \boldsymbol{0}$ (by further passing to a subsequence if necessary); hence, we have $M\bar{\bu}=\boldsymbol{0}$ by continuity. Combining \eqref{11} and \eqref{condtion}, we have $$M^T(M\bx^{2,k}-\bb)-\by+\rho (\bx^{2,k}-\bx^{1,k})=o(\delta_k).$$ 
Furthermore, $$M^T(M\bar{x}^{2,k}-\bb)-\by+\rho (\bar{x}^{2,k}-\bar{x}^{1,k})=0.$$
Subtracting the above two equalities and using \eqref{key_contradict}, we get
\begin{equation}\label{keykeykeykey}
\bx^{2,k}-\bar{\bx}^{2,k}=\bx^{1,k}-\bar{\bx}^{1,k}+o(\delta_k).
\end{equation}
Thus, $$\bar{\bu}=\lim_{k\rightarrow \infty} \frac{\bx^{2,k}-\bar{\bx}^{2,k}}{\delta^k}=\lim_{k\rightarrow \infty} \frac{\bx^{1,k}-\bar{\bx}^{1,k}}{\delta^k}.$$
Since $\bu^k \rightarrow \bar{\bu} \neq \boldsymbol{0}$, we have $\fprod{\bu^k, \bar{\bu}}>0$ for $k$ sufficiently large. Select $k$ such that $\fprod{\bu^k, \bar{\bu}}>0$ and let
\begin{equation}\label{x_hat}
\hat{\bx}^{2,k}\triangleq \bar{\bx}^{2,k}+\epsilon \bar{\bu}
\end{equation}
for some $\epsilon >0$. We can show that for $\epsilon>0$ sufficiently small 
\begin{equation}\label{belong}
    \hat{\bx}^{2,k}\in \bX^2(\by),
\end{equation}
whose proof is relegated to Appendix~\ref{app:85inLem46}. Now, assume $\hat{\bx}^{2,k}\in \bX^2(\by^k)$ for $\epsilon>0$ sufficiently small. 
This leads to the following contradiction: 
\begin{equation}
    \norm{\bx^{2,k}-\hat{\bx}^{2,k}}_2= \norm{\bx^{2,k}-\bar{\bx}^{2,k}-\epsilon\bar{\bu}}_2=\delta^k+\epsilon^2-2\epsilon \fprod{\bu^k, \bar{\bu}} < \delta^k
\end{equation}
for $\epsilon$ sufficiently small, which contradicts the definition of $\bar{\bx}^{2,k}$ in \eqref{delta}. So \eqref{key} holds.

By \eqref{eq:r1d}, we have
\begin{equation}\label{pal_1}
    \boldsymbol{0} \in \lambda \partial \sum_{g\in\cG}w_g\norm{\br^{1,k}_{j(g)}-\bx^{1,k}_{j(g)}}_2+(\br^{1,k}-\by-\rho(\bx^{1,k}-\bx^{2,k})),
\end{equation}
which is the optimal condition to
\begin{equation} \label{final_1}
    \br^{1,k}\in \argmin_\bd \lambda \sum_{g\in\cG}w_g\norm{\bd_{j(g)}-\bx^{1,k}_{j(g)}}_2+\fprod{\br^{1,k}-\by-\rho(\bx^{1,k}-\bx^{2,k}),\bd}.
\end{equation}
From \eqref{final_1}, we have
\begin{align}\label{final_2}
\begin{split}
    &\lambda\sum_{g\in\cG}w_g\norm{\br^{1,k}_{j(g)}-\bx^{1,k}_{j(g)}}_2+\fprod{\br^{1,k}-\by-\rho(\bx^{1,k}-\bx^{2,k}),\br^{1,k}} \\
	& \qquad \qquad \leq \lambda\sum_{g\in\cG}w_g\norm{\bar{\bx}^{1,k}}_2+\fprod{\br^{1,k}-\by-\rho(\bx^{1,k}-\bx^{2,k}),\bx^{1,k}-\bar{\bx}^{1,k}}.
\end{split}
\end{align}
From $\tilde{\nabla}_{\bx^1} L_{\rho}(\bar{\bx}^{1,k},\bar{\bx}^{2,k};\by)=0$, we have
\begin{equation}
    \boldsymbol{0}=\argmin_\bd \lambda\sum_{g\in\cG}w_g\norm{\bd_{j(g)}-\bar{\bx}^1_{j(g)}}_2+ \frac{1}{2}\norm{\bd-\by-\rho(\bar{\bx}^1-\bar{\bx}^2)}^2_2. \label{final_section}
\end{equation}
Similar to \eqref{pal_1}, we have 
\begin{equation}
    \boldsymbol{0} \in \lambda \partial \sum_{g\in\cG}w_g\norm{\bar{\bx}^{1,k}_{j(g)}}_2+(-\by-\rho(\bar{\bx}^{1,k}-\bar{\bx}^{2,k})),
\end{equation}
which is the optimal condition to
\begin{equation}\label{abc}
    \boldsymbol{0}\in \argmin_\bd \lambda \sum_{g\in\cG}w_g\norm{\bd_{j(g)}-\bar{\bx}^{1,k}_{j(g)}}_2+\fprod{-\by-\rho(\bar{\bx}^{1,k}-\bar{\bx}^{2,k}),\bd}.
\end{equation}
From \eqref{abc}, we have
\begin{align}\label{final_3}
\begin{split}
    &\lambda\sum_{g\in\cG}w_g\norm{\bar{\bx}^{1,k}_{j(g)}}_2\\
    &\qquad \leq \lambda\sum_{g\in\cG}w_g\norm{\br^{1,k}_{j(g)}-\bx^{1,k}_{j(g)}}_2+\fprod{-\by-\rho(\bar{\bx}^{1,k}-\bar{\bx}^{2,k}),\bar{\bx}^{1,k}+\br^{1,k}-\bx^{1,k}}.
\end{split}
\end{align}
Adding \eqref{final_2} and \eqref{final_3}, and using \eqref{r2}, we obtain
\begin{align} \label{eq:98}
\begin{split}
&\fprod{\br^{1,k}+\br^{2,k},\br^{1,k}}+\fprod{M^TM(\bx^{2,k}-\bar{\bx}^{2,k}),\bx^{1,k}-\bar{\bx}^{1,k}}\\
& \qquad \leq \fprod{\br^{1,k}+\br^{2,k},\bx^{1,k}-\bar{\bx}^{1,k}}+\fprod{M^TM(\bx^{2,k}-\bar{\bx}^{2,k}),\br^{1,k}}.
\end{split}
\end{align}
From \eqref{r2}, we have
\begin{equation} \label{eq:from67}
\bx^{1,k}-\bar{\bx}^{1,k}=(\frac{1}{\rho}M^TM+\bI)(\bx^{2,k}-\bar{\bx}^{2,k})-\frac{1}{\rho}\br^{2,k}.
\end{equation}
Defining $A\triangleq \frac{1}{\rho}M^TM+I$ and using \eqref{eq:from67} in \eqref{eq:98} and rearranging the terms, we obtain 
\begin{align}\label{QQ}
\begin{split}
&\fprod{\br^{1,k}+\br^{2,k},\br^{1,k}+\frac{1}{\rho}\br^{2,k}}+\fprod{M^TM(\bx^{2,k}-\bar{\bx}^{2,k}),A(\bx^{2,k}-\bar{\bx}^{2,k})}\\
& \qquad \leq \fprod{\br^{1,k}+\br^{2,k},A(\bx^{2,k}-\bar{\bx}^{2,k})}+\fprod{M^TM(\bx^{2,k}-\bar{\bx}^{2,k}),\br^{1,k}+\frac{1}{\rho}\br^{2,k}}. 
\end{split}
\end{align}
Let us consider term by term. We have
\begin{align}
\fprod{\br^{1,k}+\br^{2,k},\br^{1,k}+\frac{1}{\rho}\br^{2,k}} \geq \norm{\br^{1,k}}^2_2+\frac{1}{\rho}\norm{\br^{2,k}}^2_2-(\frac{1}{\rho}+1)\norm{\br^{1,k}}_2\norm{\br^{2,k}}_2.
\end{align}
Next, using \eqref{key}, we have
\begin{align}
    \fprod{M^TM(\bx^{2,k}-\bar{\bx}^{2,k}),A(\bx^{2,k}-\bar{\bx}^{2,k})}&=\frac{1}{\rho}\norm{M^TM(\bx^{2,k}-\bar{\bx}^{2,k})}_2^2+\norm{M(\bx^{2,k}-\bar{\bx}^{2,k})}^2_2\\
    &\geq \kappa^2\norm{\bx^{2,k}-\bar{\bx}^{2,k}}^2_2.
\end{align}
Denote the largest eigenvalue of matrix $A$ by $L_1$, we have
\begin{align}
\fprod{\br^{1,k}+\br^{2,k},A(\bx^{2,k}-\bar{\bx}^{2,k})} \leq L_1 \norm{\br^{1,k}+\br^{2,k}}_2 \norm{\bx^{2,k}-\bar{\bx}^{2,k}}_2.
\end{align}
Denote $L_2\triangleq \max_{\norm{\bd}=1} \norm{M\bd}$, 
\begin{align}
 \fprod{M^TM(\bx^{2,k}-\bar{\bx}^{2,k}),\br^{1,k}+\frac{1}{\rho}\br^{2,k}} \leq L_2^2\norm{\br^{1,k}+\frac{1}{\rho}\br^{2,k}}_2\norm{\bx^{2,k}-\bar{\bx}^{2,k}}_2.
\end{align}
Combining the four inequalities above, we have
\begin{align}\label{keykey}
\begin{split}
    &\kappa^2\norm{\bx^{2,k}-\bar{\bx}^{2,k}}^2_2+\norm{\br^{1,k}}^2_2+\frac{1}{\rho}\norm{\br^{2,k}}^2_2-(\frac{1}{\rho}+1)\norm{\br^{1,k}}_2\norm{\br^{2,k}}_2\\
    & \qquad \qquad \leq (L_1 \norm{\br^{1,k}+\br^{2,k}}_2+L_2^2\norm{\br^{1,k}+\frac{1}{\rho}\br^{2,k}}_2)\norm{\bx^{2,k}-\bar{\bx}^{2,k}}_2.
\end{split}
\end{align}
Denote $b\triangleq L_1 \norm{\br^{1,k}+\br^{2,k}}_2+L_2^2\norm{\br^{1,k}+\frac{1}{\rho}\br^{2,k}}_2$, $c\triangleq \norm{\br^{1,k}}^2_2+\frac{1}{\rho}\norm{\br^{2,k}}^2_2-(\frac{1}{\rho}+1)\norm{\br^{1,k}}_2\norm{\br^{2,k}}_2$. 
Using quadratic formula, \eqref{keykey} implies
\begin{equation}\label{keykeykey}
    \norm{\bx^{2,k}-\bar{\bx}^{2,k}}_2 \leq \frac{b+\sqrt{b^2-4\kappa^2 c}}{2\kappa^2}.
\end{equation}
Note that the right-hand-side of \eqref{keykeykey} is $\mathcal{O}(\norm{\br^{1,k}}+\norm{\br^{2,k}})$, 
so \eqref{keykeykey} contradicts \eqref{condtion}, which says $$\norm{\br^{1,k}}+\norm{\br^{2,k}}=o(\norm{\bx^{2,k}-\bar{\bx}^{2,k}}).$$

So far, we have shown that for a fixed $\by$, there exist $\tau$ and $\delta$ satisfying \eqref{primal_error} and the inequality below it, accordingly. From \eqref{r1} and \eqref{r2}, we know $\tilde{\nabla}_{\bx} L_{\rho}(\bx^1,\bx^2;\by)$ is continuous in $\by$. Since $\bX(\by)$ is characterized by $\tilde{\nabla}_{\bx} L_{\rho}(\bx^1,\bx^2;\by)=0$, so $\text{dist}(\bx,\bX(\by))$ is also continuous in $\by$, which implies that we can define a continuous mapping from $\by$ to $\tau$ and $\delta$. Note that from the above proof, we know that for any $\by$, $\tau$ is finite, i.e., $\tau<\infty$, and $\delta>0$. Hence, since $\by$ is in a compact set, we can find $\bar{\tau}\triangleq \sup\{\tau\}<\infty$ and $\bar{\delta}\triangleq \inf\{\delta\} >0$. This finishes the proof of the lemma.

\section{Proof of \eqref{belong} in Lemma~\ref{lem:primal_err_bound}}\label{app:85inLem46}
The following proof is inspired by \cite{tseng2010approximation} and \cite{zhang2013linear}. Denote $\bt^k\triangleq \by+\rho(\bx^{1,k}-\bx^{2,k})$. Note that if we write \eqref{eq:augLagrange} as a function of $\bx^2$ and $\bz\triangleq\bx^1-\bx^2$, then the last term is strongly convex in $\bz$. This implies that the value of $\bar{\bt}\triangleq\by+\rho(\bx^1-\bx^2)$ is unique for $\forall (\bx^1,\bx^2)\in \bX(\by)$. 
Recall the definition in \eqref{x_1} and \eqref{x_hat}, we will show \eqref{belong} is equivalent to 
\begin{equation}\label{complicate}
 \boldsymbol{0} \in \lambda \partial w_g \norm{(\bar{\bx}^{1,k}+\epsilon\bar{\bu})_{j(g)}}+\bar{\bt}_{j(g)},\ \ \forall g.
\end{equation}
From the optimality condition of \eqref{eq:augLagrange}, we know  
$\hat{\bx}^{2,k}\in \bX^2(\by)$ is equivalent to 
\begin{equation}
\begin{cases}\label{gradient}
        \boldsymbol{0} \in \lambda \partial \sum_{g\in\cG}w_g\norm{\bx^{1}_{j(g)}}_2+(\by+\rho(\bx^1-\hat{\bx}^2))_{j(g)}, \ \forall g \\
        \boldsymbol{0}=M^T(M\hat{\bx}^{2,k}-\bb)-(\by+\rho(\bx^1-\hat{\bx}^2)), \ 
\end{cases}
\end{equation}
is satisfied for some $\bx^1$.
From \eqref{x_hat}, we have 
\begin{equation}
    M\hat{\bx}^2=M\bar{\bx}^2
\end{equation}
since $M\bar{\bu}=\boldsymbol{0}$. 
So the second equality of \eqref{gradient} holds if and only if $\bx^1=\bar{\bx}^1+\epsilon\bar{\bu}$.\footnote{In fact, from our discussion on the uniqueness of $\by+\rho(\bx^1-\bx^2)$ for $\forall (\bx^1,\bx^2)\in \bX(\by)$, we can also conclude that $\bx^1$ must be $\bar{\bx}^1+\epsilon\bar{\bu}$.} Since  $\boldsymbol{0}=M^T(M\bar{\bx}^{2,k}-\bb)-(\by+\rho(\bar{\bx}^1-\bar{\bx}^2))$ holds by definitions \eqref{x_1} and \eqref{delta}, \eqref{gradient} is equivalent to 
\begin{equation}
    \boldsymbol{0} \in \lambda \partial \sum_{g\in\cG}w_g\norm{\bar{\bx}^{1}_{j(g)}+\epsilon\bar{\bu}}_2+(\by+\rho(\bar{\bx}^1-\bar{\bx}^2))_{j(g)}\ \forall g.
\end{equation}
Using $\bar{\bt}$ to replace $(\by+\rho(\bar{\bx}^1-\bar{\bx}^2))$, we have \eqref{complicate}.

Based on \eqref{condtion} and \eqref{key_contradict}, we have
\begin{equation}\label{mm}
    \bt^k-\bar{\bt}=M^TM(\bx^{2,k}-\bar{\bx}^{2,k})-\br^{2,k}=o(\delta^k).
\end{equation}
By further passing to a subsequence if necessary, we can assume that, for each $g\in\cG$, either 
\begin{enumerate}
    \item $\norm{\bx^{1,k}_{j(g)}-\bt^k_{j(g)}}_2\leq \lambda w_g, \ \ \forall k$, or, 
    \item $\norm{\bx^{1,k}_{j(g)}-\bt^k_{j(g)}}_2 > \lambda w_g$, and $\bar{\bx}^{1,k}_{j(g)}\neq \boldsymbol{0}, \ \ \forall k$, or, 
    \item $\norm{\bx^{1,k}_{j(g)}-\bt^k_{j(g)}}_2 > \lambda w_g$, and $\bar{\bx}^{1,k}_{j(g)}= \boldsymbol{0}, \ \ \forall k$,
\end{enumerate}
is true. We will show that in any of the three above cases, $\bar{\bu}_{j(g)}$ is a certain multiple of $\bar{\bt}_{j(g)}$ and then \eqref{complicate} is satisfied. 
\begin{enumerate}
    \item In this case, from \eqref{cc}, we know
      \begin{equation}
          \bar{\bu}_{j(g)}=\lim_{k\rightarrow \infty} \frac{\bx^{1,k}_{j(g)}-\bar{\bx}^{1,k}_{j(g)}}{\delta^k}=\lim_{k\rightarrow \infty} \frac{\br^{1,k}_{j(g)}-\bar{\bx}^{1,k}_{j(g)}}{\delta^k}=\lim_{k\rightarrow \infty} \frac{-\bar{\bx}^{1,k}_{j(g)}}{\delta^k}
      \end{equation}
      where the last equation comes from \eqref{condtion}.
      Suppose that $\bar{\bu}_{j(g)}\neq \boldsymbol{0}$. (Otherwise, $\hat{\bx}^{2,k}=\bar{\bx}^{2,k}$.) Then $\bar{\bx}^{1,k}_{j(g)}\neq \boldsymbol{0}$ for all $k$ sufficiently large. From the optimality condition for \eqref{eq:augLagrange}, we have
      \begin{equation}\label{pp}
          \boldsymbol{0}=\lambda w_g\frac{\bar{\bx}^{1,k}_{j(g)}}{\norm{\bar{\bx}^{1,k}_{j(g)}}_2}+ \bar{\bt}_{j(g)},
      \end{equation}
      for $k$ sufficiently large. By continuity, we have $\bar{\bu}_{j(g)}$ is a positive multiple of $\bar{\bt}_{j(g)}$. Furthermore, $\bar{\bx}^{1,k}_{j(g)}$ is a negative multiple of $\bar{\bt}_{j(g)}$. Therefore, for $\epsilon$ sufficiently small, \eqref{complicate} is satisfied. 
    \item In this case, since we assumed $\bar{\bx}^{1,k}_{j(g)}\neq \boldsymbol{0} \ \ \forall k$, \eqref{pp} is always satisfied. It implies 
    \begin{equation}\label{4_3}
        \bar{\bt}_{j(g)}=\lambda w_g\frac{\bar{\bt}_{j(g)}-\bar{\bx}^{1,k}_{j(g)}}{\norm{\bar{\bt}_{j(g)}-\bar{\bx}^{1,k}_{j(g)}}_2}.
    \end{equation}
From \eqref{cc}, we have
     \small{
     \begin{align*}
         \br^{1,k}_{j(g)}&=\frac{\lambda w_g}{\norm{\bx^{1,k}_j(g)-\bt^k_{j(g)}}_2}\bx^{1,k}_{j(g)}+(\frac{\norm{\bx^{1,k}_j(g)-\bt^k_{j(g)})}_2-\lambda w_g}{\norm{\bx^{1,k}_j(g)-\bt_{j(g)}^k}_2})\bt^k_{j(g)}\\
         &=\frac{\lambda w_g}{\norm{\bx^{1,k}_j(g)-\bt^k_{j(g)}}_2}(\bar{\bx}^{1,k}_{j(g)}+\delta^k\bu^k_{j(g)}+o(\delta^k))+(\frac{\norm{\bx^{1,k}_j(g)-\bt^k_{j(g)})}_2-\lambda w_g}{\norm{\bx^{1,k}_{j(g)}-\bt_{j(g)}^k}_2})(\bar{\bt}_{j(g)}+o(\delta^k))\\
         &=\frac{\lambda w_g \delta^k}{\norm{\bx^{1,k}_{j(g)}-\bt^k_{j(g)}}_2}\bu^k_{j(g)}+\frac{\lambda w_g}{\norm{\bx^{1,k}_{j(g)}-\bt^k_{j(g)}}_2}(\bar{\bx}^{1,k}_{j(g)}-\bar{\bt}_{j(g)})+\bar{\bt}_{j(g)}+o(\delta^k)\\
         &=\frac{\lambda w_g \delta^k}{\norm{\bx^{1,k}_{j(g)}-\bt^k_{j(g)}}_2}\bu^k_{j(g)}+(\frac{\lambda w_g}{\norm{\bar{\bt}_{j(g)}-\bar{\bx}^{1,k}_{j(g)}}_2}-\frac{\lambda w_g}{\norm{\bt^k_{j(g)}-\bx^{1,k}_{j(g)}}_2})(\bar{\bt}_{j(g)}-\bar{\bx}^{1,k}_{j(g)})+o(\delta^k)\\
         &=\frac{\lambda w_g \delta^k}{\norm{\bar{\bt}_{j(g)}-\bar{\bx}^{1,k}_{j(g)}-\delta^k\bu^k_{j(g)}+o(\delta^k)}_2}\bu^k_{j(g)}\\
         &+(\frac{\lambda w_g}{\norm{\bar{\bt}_{j(g)}-\bar{\bx}^{1,k}_{j(g)}}_2}-\frac{\lambda w_g}{\norm{\bar{\bt}_{j(g)}-\bar{\bx}^{1,k}_{j(g)}-\delta^k\bu^k_{j(g)}+o(\delta^k)}_2})(\bar{\bt}_{j(g)}-\bar{\bx}^{1,k}_{j(g)})+o(\delta^k)
     \end{align*}
}
where the second equality comes from \eqref{keykeykeykey} and \eqref{mm}. The forth equality follows from \eqref{4_3}. Finally, we use \eqref{keykeykeykey} and \eqref{mm} in the last equality. From the Taylor expansion of $\norm{\cdot}_2^{-1}$ and given that $\nabla_{\bx} \norm{\bx}_2^{-1}=-\bx/\norm{\bx}^3_2$, we have
     \begin{align*}
         &\frac{1}{\norm{\bar{\bt}_{j(g)}-\bar{\bx}^{1,k}_{j(g)}-\delta^k\bu^k_{j(g)}+o(\delta^k)}_2}=\frac{1}{\norm{\bar{\bt}_{j(g)}-\bar{\bx}^{1,k}_{j(g)}}_2}-\frac{\fprod{\bar{\bt}_{j(g)}-\bar{\bx}^{1,k}_{j(g)},-\delta^k\bu^k_{j(g)}+o(\delta^k)}}{\norm{\bar{\bt}_{j(g)}-\bar{\bx}^{1,k}_{j(g)}}_2^3}\\
         &+o(\norm{-\delta^k\bu^k_{j(g)}+o(\delta^k)}_2)\\
         &=\frac{1}{\norm{\bar{\bt}_{j(g)}-\bar{\bx}^{1,k}_{j(g)}}_2}+\frac{\fprod{\bar{\bt}_{j(g)}-\bar{\bx}^{1,k}_{j(g)},\delta^k\bu^k_{j(g)}}}{\norm{\bar{\bt}_{j(g)}-\bar{\bx}^{1,k}_{j(g)}}_2^3}+o(\delta^k).
     \end{align*}
     Using this back in the last equation for $\br^{1,k}_{j(g)}$ and rearranging the terms, we have
     \begin{align*}
       \br^{1,k}_{j(g)}&=\frac{\lambda w_g\delta^k}{\norm{\bar{\bt}_{j(g)}-\bar{\bx}^{1,k}_{j(g)}}_2}\bu^k_{j(g)}-\frac{\lambda w_g\fprod{\bar{\bt}_{j(g)}-\bar{\bx}^{1,k}_{j(g)},\delta^k\bu^k_{j(g)}}}{\norm{\bar{\bt}_{j(g)}-\bar{\bx}^{1,k}_{j(g)}}_2^3} (\bar{\bt}_{j(g)}-\bar{\bx}^{1,k}_{j(g)})+o(\delta^k)\\
       &=\frac{\lambda w_g\delta^k}{\norm{\bar{\bt}_{j(g)}-\bar{\bx}^{1,k}_{j(g)}}_2}\bu^k_{j(g)}-\frac{\fprod{\bar{\bt}_{j(g)},\delta^k\bu^k_{j(g)}}}{\norm{\bar{\bt}_{j(g)}-\bar{\bx}^{1,k}_{j(g)}}_2} (\frac{\bar{\bt}_{j(g)}}{\lambda w_g})+o(\delta^k),
     \end{align*}
where the second equality uses \eqref{4_3}. Multiplying both sides by $\frac{\norm{\bar{\bt}_{j(g)}-\bar{\bx}^{1,k}_{j(g)}}_2}{\lambda w_g \delta^k}$ and using \eqref{condtion},\eqref{limitt} and $\norm{\bar{\bt}_{j(g)}}_2=\lambda w_g$ (from \eqref{4_3}) yields in the limit 
     \begin{equation}\label{case_2}
         \boldsymbol{0}=\bar{\bu}_{j(g)}-\frac{\fprod{\bar{\bt}_{j(g)},\bar{\bu}_{j(g)}}}{\norm{\bar{\bt}_{j(g)}}^2_2}\bar{\bt}_{j(g)}.
     \end{equation}
     Thus $\bar{\bu}_{j(g)}$ is a nonzero multiple of $\bar{\bt}_{j(g)}$. In this case, since we assume $\bar{\bx}^{1,k}_{j(g)}\neq \boldsymbol{0}$, from \eqref{pp}, we know $\bar{\bx}^{1,k}_{j(g)}$ is a negative multiple of $\bar{\bt}_{j(g)}$. So \eqref{complicate} is satisfied for $\epsilon$ sufficiently small. 
    \item In this case, we assume $\bar{\bx}^{1,k}_{j(g)}= \boldsymbol{0}, \forall k$, from \eqref{limitt}, we have $\bar{\bx}^1_{j(g)}=\boldsymbol{0}$. We also assume that $\norm{\bx^{1,k}_{j(g)}-\bt^k_{j(g)}}_2 > \lambda w_g$ for all $k$, this implies $\norm{\bar{\bt}_{j(g)}}_2\geq \lambda w_g$. From the optimality condition for \eqref{eq:augLagrange} for $\bx^1$ we have
    \begin{equation*}
        \boldsymbol{0}=\bar{\bt}_{j(g)}+\lambda w_g \partial\norm{\boldsymbol{0}}_2,
    \end{equation*}
    which implies $\norm{\bar{\bt}_{j(g)}}_2\leq \lambda w_g$. Thus $\norm{\bar{\bt}_{j(g)}}_2=\lambda w_g$. Then \eqref{cc} implies 
    \begin{align*}
    \br^{1,k}_{j(g)}&=\frac{\lambda w_g}{\norm{\bx^{1,k}_j(g)-\bt^k_{j(g)}}_2}\bx^{1,k}_{j(g)}+(\frac{\lambda w_g}{\norm{\bar{\bt}_{j(g)}}_2}-\frac{\lambda w_g}{\norm{\bx^{1,k}_j(g)-\bt_{j(g)}^k}_2})\bt^k_{j(g)}\\
    &=\frac{\lambda w_g}{\norm{\bx^{1,k}_j(g)-\bt^k_{j(g)}}_2}\bx^{1,k}_{j(g)}+\frac{\lambda w_g \fprod{\bar{\bt}_{j(g)}, \bt^k_{j(g)}-\bx^{1,k}_{j(g)}-\bar{\bt}_{j(g)}}}{\norm{\bar{\bt}_{j(g)}}^3_2}\bt_{j(g)}\\
    &+o(\norm{\bt^k_{j(g)}-\bx^{1,k}_j(g)-\bar{\bt}_{j(g)}}_2)\\
    &=\frac{\lambda w_g}{\norm{\bx^{1,k}_j(g)-\bt^k_{j(g)}}_2}\bx^{1,k}_{j(g)}-\frac{\lambda w_g \fprod{\bar{\bt}_{j(g)}, \bx^{1,k}_{j(g)}}}{\norm{\bar{\bt}_{j(g)}}^3_2}\bt_{j(g)}+o(\delta^k)
    \end{align*}
    where the second equality uses Taylor expansion similar to the case 2. The third equality follows from \eqref{mm} and $\{\bx^{1,k}_{j(g)}\}\rightarrow\boldsymbol{0}$. Dividing both sides by $\delta^k$ yield in the limit \eqref{case_2}, where it uses $$\{\frac{\bx^{1,k}_{j(g)}}{\delta^k}\}=\{\bu^k_{j(g)}+\frac{o(\delta^k)}{\delta^k}\}\rightarrow \bar{\bu}_{j(g)}.$$\\
    Since we assume $\norm{\bx^{1,k}_{j(g)}-\bt^k_{j(g)}}_2 > \lambda w_g$ for all $k$, we have the following equality from \eqref{r1} 
    \begin{equation}\label{122}
        \boldsymbol{0}=\lambda w_g \frac{\br^{1,k}_{j(g)}-\bx^{1,k}_{j(g)}}{\norm{\br^{1,k}_{j(g)}-\bx^{1,k}_{j(g)}}_2}+\br^{1,k}_{j(g)}-\bt^k_{j(g)}.
    \end{equation}
        Suppose $\bar{\bu}_{j(g)}\neq \boldsymbol{0}$. Then $\bu^k_{j(g)}=\frac{\bx^{1,k}_{j(g)}}{\delta^k}+\frac{o(\delta^k)}{\delta^k}\neq \boldsymbol{0}$, for $k$ sufficiently large. It implies that $\bx^{1,k}_{j(g)}\neq \boldsymbol{0}$, for $k$ sufficiently large. Hence,
        
    \begin{align*}
        \fprod{\bar{\bt}_{j(g)},\bar{\bu}_{j(g)}}&=\lim_{k\rightarrow +\infty} \fprod{\bt^k_{j(g)},\bu^k_{j(g)}}\\
        &=\lim_{k\rightarrow +\infty}\fprod{\br^{1,k}_{j(g)},\frac{\bx^{1,k}_{j(g)}}{\delta^k}}+\fprod{\lambda w_g\frac{\br^{1,k}_{j(g)}-\bx^{1,k}_{j(g)}}{\norm{\br^{1,k}_{j(g)}-\bx^{1,k}_{j(g)}}_2},\frac{\bx^{1,k}_{j(g)}}{\delta^k}}\\
        &=\lim_{k\rightarrow +\infty}\frac{\lambda w_g}{\norm{\br^{1,k}_{j(g)}-\bx^{1,k}_{j(g)}}_2}(\frac{\fprod{\br^{1,k}_{j(g)},\bx^{1,k}_{j(g)}}}{\delta^k}-\frac{\norm{\bx^{1,k}_{j(g)}}_2^2}{\delta^k})\\
        &=\lim_{k\rightarrow +\infty}\frac{\lambda w_g}{\norm{\frac{\br^{1,k}_{j(g)}}{\norm{\bx^{1,k}_{j(g)}}_2}-\frac{\bx^{1,k}_{j(g)}}{\norm{\bx^{1,k}_{j(g)}}_2}}_2}(\fprod{\frac{\br^{1,k}_{j(g)}}{\delta^k},\frac{\bx^{1,k}_{j(g)}}{\norm{\bx^{1,k}_{j(g)}}_2}}-\norm{\bu^k_{j(g)}}_2)\\
        &=-\lambda w_g \norm{\bu^k_{j(g)}}_2 <0,
    \end{align*}
where the second equality is based on \eqref{122} and $\lim_{k\rightarrow +\infty} \bu^k_{j(g)}=\lim_{k\rightarrow +\infty}\frac{\bx^{1,k}_{j(g)}}{\delta^k}$, the third equality is based on $\frac{\br^{1,k}_{j(g)}}{\delta^k}\rightarrow \boldsymbol{0}$ (by \eqref{condtion}), the forth equality is based on $\bx^{1,k}_{j(g)}\neq \boldsymbol{0}$ and $\lim_{k\rightarrow +\infty} \bu^k_{j(g)}=\lim_{k\rightarrow +\infty}\frac{\bx^{1,k}_{j(g)}}{\delta^k}$, and, the fifth equality is based on $\br^{1,k}_{j(g)}\rightarrow \boldsymbol{0}$ and $\frac{\br^{1,k}_{j(g)}}{\delta^k}\rightarrow \boldsymbol{0}$ (by \eqref{condtion}). Finally, combined with \eqref{case_2}, we obtain $\bar{\bu}_{j(g)}$ is a negative multiplier of $\bar{\bt}_{j(g)}$. Since $\bar{\bx}^1_{j(g)}=\boldsymbol{0}$ in this case, \eqref{complicate} is satisfied. 
\end{enumerate}

\section{Proof of Theorem~\ref{thm:main}}\label{app:main_thm}
From Lemmas~2.4 and 2.5 in \cite{hong2017linear}, we have the following two identities, respectively:
\begin{align}
    L(\bx^k; \by^k)-L(\bx^{k+1}; \by^k)&\geq \rho \norm{\bx^k-\bx^{k+1}}^2,\label{L_difference}\\
    \norm{\tilde{\nabla}L(\bx^k; \by^k)}&\leq \sigma\norm{\bx^k-\bx^{k+1}}\label{L_gradient},
\end{align}
where $\sigma=\sqrt{2}(\max \{1+\norm{M^T}\norm{M}, \rho\}+1)$.

Lemma~\ref{lem:boundedness} in our paper establishes the uniform boundedness of iterates $\{(\bx^k,\by^k)\}$, based on which the compactness condition of Lemma~\ref{lem:primal_err_bound} is satisfied. Lemma~\ref{lem:primal_err_bound} quantifies the primal error bound with the proximal gradient of the Lagrangian function as
\begin{equation}
	\norm{\bx^k-\bar{\bx}^k}\leq\tau_p\norm{\tilde{\nabla}L(\bx^k; \by^k)}\leq\tau_p\sigma\norm{\bx^k-\bx^{k+1}},\label{x_sigma}
\end{equation}
where $\bar{\bx}^k=\argmin_{\bar{\bx}\in\bX(\by^k)}  \norm{\bar{\bx}-\bx^k}$ and the second inequality comes from \eqref{L_gradient}.

In Lemma~\ref{lem:dual_err_bound}, we show
\begin{equation}\label{main:dual_error_bound}
  \text{dist}(\by,Y^*)\leq \tau_d\norm{\grad g_{\rho}(\by)}_2,
\end{equation}
where $\tau_d=\max\{\norm{M^T}^2, 2\rho\}$ as shown in the proof of Lemma \ref{lem:dual_err_bound}. Based on the dual error bound \eqref{main:dual_error_bound} and following Lemma~3.1 in \cite{hong2017linear}, we have 
\begin{align}
    &\Delta^k_d\leq \tau'\norm{\nabla g_{\rho}(\by^k)}^2=\tau'\norm{\bar{\bx}^{1,k}-\bar{\bx}^{2,k}}^2, \label{d}\\
    &\Delta^k_p\leq \zeta\norm{\bx^k-\bx^{k+1}}^2+\zeta'\norm{\bx^k-\bar{\bx}^k}^2\leq(\xi+\xi'\tau_p^2\sigma^2)\norm{\bx^{k+1}-\bx^k}^2,\label{zeta_two}
\end{align}
where $\bar{\bx}^{1,k}$ and $\bar{\bx}^{2,k}$ represent the upper half and lower half of the vector $\bar{\bx}^k$, correspondingly, $\tau'=\tau_d^2/\rho$, and
\begin{align}
	\zeta&=2A+\frac{3\sqrt{2}}{2}(\sigma-1),\\
	\zeta'&=2A+\frac{1}{2}+\frac{\sqrt{2}}{2}(\sigma-1),
\end{align}
where $A=\norm{M^T}\norm{M}+\sqrt{2}\rho$.

Following Lemmas~3.2 and 3.3 in \cite{hong2017linear}, we have
\begin{align}
	\Delta_d^k-\Delta_d^{k-1}&\leq-\alpha(\bx^{1,k}-\bx^{2,k})^T(\bar{\bx}^{1,k}-\bar{\bx}^{2,k}),\label{delta_dual}\\
	\Delta_p^k-\Delta_p^{k-1}&\leq \alpha\norm{\bx^{1,k}-\bx^{2,k}}^2-\gamma\norm{\bx^{k+1}-\bx^k}^2-\alpha(\bx^{1,k}-\bx^{2,k})^T(\bar{\bx}^{1,k}-\bar{\bx}^{2,k}),\label{delta_primal}
\end{align}
where $\Delta_p^k$ and $\Delta_d^k$ are primal and dual optimality gaps at iteration $k$ (defined in the statement of Theorem~\ref{thm:main}), $\alpha$ is the stepsize, and in \eqref{delta_primal}, we have used \eqref{L_difference}. Adding \eqref{delta_dual} and \eqref{delta_primal}, we have 
\begin{align}
	[\Delta_d^k+\Delta_p^k]-[\Delta_d^{k-1}+\Delta_p^{k-1}]&\leq \alpha\norm{\bx^{1,k}-\bx^{2,k}}^2-\gamma\norm{\bx^{k+1}-\bx^k}^2-2\alpha(\bx^{1,k}-\bx^{2,k})^T(\bar{\bx}^{1,k}-\bar{\bx}^{2,k})\\
	&=\alpha\norm{\bx^{1,k}-\bx^{2,k}-\bar{\bx}^{1,k}+\bar{\bx}^{2,k}}^2-\alpha\norm{\bar{\bx}^{1,k}-\bar{\bx}^{2,k}}^2-\rho\norm{\bx^{k+1}-\bx^k}^2\\
	&\leq (2\alpha\tau_p^2\sigma^2-\rho)\norm{\bx^{k+1}-\bx^k}^2-\alpha\norm{\bar{\bx}^{1,k}-\bar{\bx}^{2,k}}^2,\label{key_main}
\end{align}
where the last inequality comes from \eqref{x_sigma} and the Cauchy-Schwarz inequality.

Assuming that the stepsize $\alpha$ is chosen sufficient small such that $0<\alpha<\frac{\rho}{2\tau_p^{2}\sigma^{2}}$, and substituting \eqref{d} and \eqref{zeta_two} into \eqref{key_main}, we have 
\begin{align}
	[\Delta_d^k+\Delta_p^k]-[\Delta_d^{k-1}+\Delta_p^{k-1}]&\leq -\frac{\rho-2\alpha\tau_p^2\sigma^2}{\xi+\xi'\tau_p\sigma^2}\Delta_p^k-\frac{\alpha}{\tau'}\Delta_d^k\\
	&\leq -\min\{\frac{\rho-2\alpha\tau_p^2\sigma^2}{\xi+\xi'\tau_p\sigma^2},\frac{\alpha}{\tau'}\}[\Delta_d^k+\Delta_p^k].
\end{align}

Therefore, we have
\begin{equation}\label{main:result}
    0\leq [\Delta^k_p+\Delta^k_d]\leq \frac{1}{\lambda+1}[\Delta^{k-1}_p+\Delta^{k-1}_d],
\end{equation}
where $\lambda=\min\{\frac{\rho-2\alpha\tau_p^2\sigma^2}{\zeta+\zeta'\tau_p^2\sigma^2},\frac{\alpha}{\tau'}\}>0$. 
Therefore, $[\Delta^k_p+\Delta^k_d]\leq (\frac{1}{\lambda+1})^k[\Delta^0_p+\Delta^0_d]$, implying that $\Delta^k_p$ and $\Delta^k_d$ converges to zero Q-linearly.

\end{document}